\documentclass[A4]{amsart}
\input xypic
\input xy \xyoption{all}

\let\originallhook\lhook
\let\originalrhook\rhook
\let\lhook\originallhook
\let\rhook\originalrhook

\usepackage{amssymb}
\usepackage{bbm}
\usepackage{amsmath}

\newtheorem{theorem}{Theorem}[section]
\newtheorem{proposition}[theorem]{Proposition}
\newtheorem{lemma}[theorem]{Lemma}
\newtheorem{corollary}[theorem]{Corollary}

\theoremstyle{definition}
\newtheorem{definition}[theorem]{Definition}
\newtheorem{bigremark}[theorem]{Remark}

\newtheorem{example}[theorem]{Example}

\newcommand{\ga}[1]{m_{#1}}

\newcommand{\abSc}{ab\#c}
\newcommand{\bcSa}{bc\#a}

\newcounter{bean}

\newcommand{\incl}[3]{\ensuremath{#1\stackrel{#2}
 {\lhook\joinrel\longrightarrow}#3}}

\newcommand{\seqm}[3]{\ensuremath{#1\stackrel{#2}
 {\longrightarrow}#3}}
\newcommand{\seqmm}[5]{\ensuremath{#1\stackrel{#2}
 {\longrightarrow}#3\stackrel{#4}{\longrightarrow}#5}}
\newcommand{\seqmmm}[7]{\ensuremath{#1\stackrel{#2}
 {\longrightarrow}#3\stackrel{#4}{\longrightarrow}#5
  \stackrel{#6}{\longrightarrow}#7}}

\newcommand{\mapaddr}[2]{\ensuremath{\stackrel{#1}{\longrightarrow}#2}}

\newcommand{\incladdr}[2]{\ensuremath{\stackrel{#1}{\lhook\joinrel\longrightarrow}#2}}

\newcommand{\ceil}[1]{\ensuremath{\left\lceil #1 \right\rceil}}
\newcommand{\paren}[1]{\ensuremath{\left( #1 \right)}}
\newcommand{\br}[1]{\ensuremath{\left\{ #1 \right\}}}
\newcommand{\vbr}[1]{\ensuremath{\left\langle#1\right\rangle}}

\newcommand{\cvee}[3]{\displaystyle\bigvee^{#2}_{#1}#3}

\newcommand{\cplus}[3]{\displaystyle\bigoplus^{#2}_{#1}#3}
\newcommand{\cdunion}[3]{\displaystyle\coprod^{#2}_{#1}#3}
\newcommand{\csqdunion}[3]{\displaystyle\bigsqcup^{#2}_{#1}#3}

\newcommand{\csum}[3]{\displaystyle\sum^{#2}_{#1}#3}

\newcommand{\cunionmulti}[4]{\displaystyle\bigcup^{#3}
_{\renewcommand{\arraystretch}{0.6}\begin{matrix}\scriptstyle #1 \cr \scriptstyle #2\end{matrix}\renewcommand{\arraystretch}{1.2}}#4}
\newcommand{\cdunionmulti}[4]{\displaystyle\coprod^{#3}
_{\renewcommand{\arraystretch}{0.6}\begin{matrix}\scriptstyle #1 \cr \scriptstyle #2\end{matrix}\renewcommand{\arraystretch}{1.2}}#4}

\newcommand{\cset}[2]{\br{#1\,\,\middle\vert\,\,#2}}
\newcommand{\indexsize}[1]{\renewcommand{\arraystretch}{0.6}#1\renewcommand{\arraystretch}{1.2}}

\newcommand{\qqed}{\hfill\square}

\newcommand{\sm}{\backslash}

\newcommand{\mc}[1]{\ensuremath{\mathcal{#1}}}
\newcommand{\mb}[1]{\ensuremath{\mathbb{#1}}}

\newcommand{\Z}{\ensuremath{\mathbb{Z}}}

\newcommand{\ID}{\ensuremath{\mathbbm{1}}}

\newcommand{\bd}{\ensuremath{\partial}}

\newcommand{\wcolon}{\ensuremath{\,\colon\,}}

\DeclareMathOperator{\Nill}{nill}

\DeclareMathOperator{\cat}{cat}
\DeclareMathOperator{\Cat}{\mc Cat}
\DeclareMathOperator{\gcat}{gcat}
\DeclareMathOperator{\cuplen}{cup}
\DeclareMathOperator{\link}{link}
\DeclareMathOperator{\filt}{filt}
\DeclareMathOperator{\hd}{hd}
\DeclareMathOperator{\hc}{hc}
\def\Tor{\mathrm{Tor}}

\newcommand{\sk}[2]{{#1}^{\paren{#2}}}
\newcommand{\z}[1]{\mathcal Z_{#1}}
\newcommand{\zsk}[2]{\mathcal Z_{{#1}^{\paren{#2}}}}
\newcommand{\zk}{\mathcal Z_K}
\newcommand{\zl}{\mathcal Z_L}
\newcommand{\hz}[1]{\widehat{\mathcal Z}_{#1}}
\newcommand{\hzk}{\widehat{\mathcal Z}_K}

\newcommand{\rzsk}[2]{\mathbb R\mathcal Z_{{#1}^{\paren{#2}}}}
\newcommand{\rzk}{\mathbb R\mathcal Z_K}

\begin{document}
\title[$LS$-Category of Moment-Angle Manifolds]{$LS$-category of moment-angle manifolds and higher order Massey products}

\author{Piotr Beben}
\address{\scriptsize{School of Mathematics, University of Southampton,Southampton SO17 1BJ, United Kingdom}} 
\email{P.D.Beben@soton.ac.uk} 
\author{Jelena Grbi\'c} 
\address{\scriptsize{School of Mathematics, University of Southampton,Southampton SO17 1BJ, United Kingdom}}  
\email{J.Grbic@soton.ac.uk} 

\subjclass[2010]{Primary  55M30, 55U10, Secondary 32Q55, 05E40}
\keywords{polyhedral product, moment-angle complex, toric topology, Massey products, Lusternik-Schnirelmann category, non-K\" ahler manifolds} 

\begin{abstract}
Using the combinatorics of the underlying simplicial complex $K$, we give various upper and lower bounds for the Lusternik-Schnirelmann (LS) category of moment-angle complexes $\zk$. We describe families of simplicial complexes and combinatorial operations which allow for a systematic description of the LS category. In particular, we characterise the LS category of moment-angle complexes $\zk$ over triangulated $d$-manifolds $K$ for $d\leq 2$, 
as well as higher dimension spheres built up via connected sum, join, and vertex doubling operations. 
We show that the LS category closely relates to vanishing of Massey products in $H^*(\zk)$ and through this connection we describe first structural properties of Massey products in moment-angel manifolds.
Some of further applications include calculations of the LS category and the description of conditions for vanishing of Massey products for moment-angle manifolds over fullerenes, Pogorelov polytopes and $k$-neighbourly complexes, which double as important examples of hyperbolic manifolds. 
\end{abstract}
\maketitle

\section{Introduction}

A covering of a topological space $X$ is said to be \emph{categorical} if every set in the covering is open and contractible in $X$, that is, the inclusion map of each set into $X$ is nullhomotopic.
The \emph{Lusternik-Schnirelmann category} (or simply \emph{category}) $\cat(X)$ of $X$ is the smallest integer $k$
such that $X$ admits a categorical covering by $k+1$ open sets $\{U_0,\ldots,U_k\}$.

In general it is not easy to compute these invariants. Lusternik and Schnirelmann~\cite{LS1, LS2} introduced it initially in connection to variational problems. Poincar\' e studying dynamical systems suggested that the existence and form of solutions of differential equations, that is, the complexity of flows should be related to the topological complexity of the underlying manifold. In this context, a first step was to estimate the number of invariant points for the particular case of gradient flows, equivalently, to estimate the minimal number of critical points of functions on the manifold.  The Lusternik-Schnirelmann invariant, nowadays known as the LS category, gives a lower bound on the number of critical points for any smooth function on the manifold. While this was analytical in nature, there has been a host of useful applications in geometry and algebraic topology. For example, G. Whitehead~\cite[page 464]{MR516508} showed that for a group-like space $G$, $\cat(X)\geq \mathrm{nilp}([X, G])$. 

In this paper, motivated by the study of the Lusternik-Schnirelmann category and related invariants for polyhedral products of the form $(X,\ast)^K$ 
for certain nice spaces $X$ in~\cite{MR2457428,MR3642766}, 
we are focusing on the LS category of moment-angle complexes $\zk$ which are a particular example of polyhedral product $(X,A)^K$ for $X=D^2$ and $A=S^1$. 

A large portion of the attention that moment-angle complexes have received has been due to their relevance to algebraic and complex geometry, combinatorics and algebra, in particular homology of local rings. The relation between cohomology of moment-angle complexes and homology of local rings can be seen through the notion of Golod rings. To a simplicial complex $K$ on $n$ vertices, the Stanley-Reisner algebra $R[K]= R[v_1,\ldots, v_n]/I_K$, where $R$ is a commutative ring with unit, can be associated and there is a ring isomorphism $H^*(\zk; R)\cong \mathrm{Tor}_{R[v_1,\ldots, v_n]}(R[K], R)$.  In particular, the case of $\zk$ being a co-$H$-space, that is, $\cat(\zk)=1$,  is closely related to the Stanley-Reisner ring $R[K]$ being a Golod ring. 

Let $R$ be a commutative Noetherian local ring with maximal ideal $\mathfrak m$ and residue field $k = R/\mathfrak m$. Let M be a finitely generated module over $R$. The Poincar\' e series $P^R_M(t)$ of $M$ over $R$ is a formal power series in $\Z[|t|]$ defined as
\[
P^R_M(t)=\sum_{i\geq 0} \beta^R_i(M)t^i\in\Z[|t|]
\]
where $\beta^R_i (M) = \dim_k \Tor_R^i (M, k)$ denotes the $i$-th Betti number of $M$. In the 1950s Serre and Kaplansky asked whether the Poincar\' e series $P^R_k (t)$ is a rational function. This question ties in closely with an analogous question in algebraic topology attributed to Serre on the rationality of the
Poincar\' e series of loop spaces of finite simply connected CW complexes (see~\cite{MR563510}). Several authors
made an attempt to find an affirmative answer. However in 1982 Anick found a counterexample
answering the question in negative~\cite{MR644015}. Research on this topic intensified since the appearance of
Anick's example.
The question of the Stanley-Reisner ring $R[K]$ being Golod translates to cup products and all higher Massey products vanishing in $H^*(\zk)$. In fact, there is a fairly large literature that is focused on determining those $K$ for which the Stanley-Reisner ring $R[K]$ is Golod (see~\cite{MR3658721, MR3633129, MR3461047,MR3084441,MR2321037,2013arXiv1306.6221I,BerglundRational}).

Away from algebra, moment-angle complexes have found intimate connections and applications to complex and symplectic geometry. 
For example, Bosio and Meersseman~\cite{MR2285318} studied a class of complete intersections of real quadrics
in $\mathbb C^n$ called links and showed that all links (after taking products with circles in odd dimensional cases) can be given a complex-analytic structure. Some of them are non-K\" ahler complex manifolds which generalise the class of Hopf and Calabi-Ekmann manifolds. Crucially, Bosio and Meersseman established a connection
between complex geometry and toric topology by showing that links coincide with moment-angle manifolds
$\mathcal Z_P$ over a simple polytope $P$. 
Deligne, Griffiths, Morgan and Sullivan~\cite{MR0382702} proved that  compact K\"ahler manifolds are formal.  Therefore they have trivial Massey products with field coefficients. 
In this way, moment-angle manifolds admitting non-trivial Massey products 
describe some of the few known explicit families of complex manifolds beyond K\"ahler manifolds. This is one of the main motivations to study Massey products and more generally 
topology of moment-angle manifolds.

The above two questions; Golodness of Stanley-Reisner rings and Massey products in complex manifolds, provide a motivation to study the LS category and its relations to Massey products of  the large class of moment-angle complexes known as  \emph{moment-angle manifolds}. Moment-angle complexes $\zk$ take the form of topological manifolds when $K$ is a simplicial sphere or equivalently when moment-angle complexes are considered over simple polytopes.
Their topology and cohomology is very intricate, with many questions remaining open even for low dimensional $K$ 
(see for example~\cite{MR3363157,MR2330154,MR3073929,MR1104531,MR531977,MR2285318,MR3593998,MR3507477}).

In Section 3 we describe filtrations of simplicial complexes $K$ which imply lower bounds on the LS-category of $\zk$. We continue by detecting combinatorial operations on simplicial complexes which in contrast to filtrations imply upper bounds on $\cat(\zk)$. 
As specific applications of these methods in Section 4, we characterise  
triangulations $K$ of $d$-spheres where $d\leq 2$ for which $\zk$ has a given category, and 
we build up higher dimensional spheres $K$ for which we can calculate $LS(\zk)$.

In Proposition~\ref{catsurface} we further extend those results by calculating $\cat(\zk)$ for any triangulation of a closed oriented surface $K$. Up to now only $\cat(\zk))=1$ has been considered.

Our further motivation is going back to the homology of rings and the study of a combinatorial and algebraic characterisation of Golod complexes $K$ and $co$-$H$-space 
(category $\leq 1$) moment-angle complexes $\zk$ given in~\cite{MR3461047} in the case of flag complexes $K$.
The authors there showed that both of these concepts are equivalent, and moreover, that they both coincide with chordality 
of the $1$-skeleton of $K$ and the triviality of the multiplication on $H^*(\zk; R)\cong \Tor^+_{R[v_1,\ldots,v_n]}(R[K],R)$
for $R=\mb Z$ or $R$ any field. 
An interesting consequence of this from the perspective of commutative algebra was that for $K$ flag the trivial multiplication on 
$\Tor^+_{R[v_1,\ldots,v_n]}(R[K],R)$ implies that all higher Massey products are also trivial.
This depended on the general fact that the cohomology ring of a space of category less than equal to 1 has trivial multiplication and Massey products 
vanish~\cite{MR0231375,MR0356037}.
It is natural to ask what the corresponding statement is for spaces with larger category, 
more so, if the characterisation for Golod flag complexes in~\cite{MR3461047} can be generalised in terms of Massey products. 
An answer to the first question was given by Rudyak in~\cite{MR1644063}, which inspired us to give the following definition.

\begin{definition}
A simplicial complex $K$ on vertex set $[n]$ is $m$-\emph{annihilating} over $R$ if
\begin{itemize}
\item[(1)] $\Nill(\mathrm{Tor}_{R[v_1,\ldots,v_n]}(R[K],R))\leq m+1$;
\item[(2)] Massey products $\vbr{v_1,\ldots,v_k}$ vanish in $\mathrm{Tor}^+_{R[v_1,\ldots,v_n]}(R[K],R)$ 
whenever $v_i=a_1\cdots a_{m_i}$ and $v_j=b_1\cdots b_{m_j}$, and $m_i+m_j>m$ for some odd $i$ and even $j$ and
$a_s,b_t\in\mathrm{Tor}^+_{R[v_1,\ldots,v_n]}(R[K],R)$.
\end{itemize}
\end{definition}

\begin{proposition}
\label{PRudyak1}
If $\cat(\zk)\leq m$, then $K$ is $m$-annihilating.~$\qqed$ 
\end{proposition}

Here the \emph{nilpotency} $\Nill A$ of a graded algebra $A$ is the smallest integer $k$ such that all length $k$ products in the 
positive degree part $A^+$ vanish. 
Notice that $K$ is $(m+1)$-annihilating whenever it is $m$-annihilating, and $1$-annihilating of $K$ coincides with $K$ being
Golod~\cite{MR0138667}, namely, that all products and (higher) Massey products are trivial in $\mathrm{Tor}^+_{R[v_1,\ldots,v_n]}(R[K],R)$. 
All of this can be restated equivalently in terms of the cohomology of $\zk$ due to an isomorphism of graded commutative algebras
$H^*(\zk;R)\cong\mathrm{Tor}_{R[v_1,\ldots,v_n]}(R[K],R)$ when $R$ is a field or $\Z$~\cite{MR2117435}.
We shall consider (co)homology with integer coefficients. 

We note that the converse of Proposition~\ref{PRudyak1} might not be true in general. For example, 
Iriye and Yano in~\cite{MR3649878} constructed an example of a Golod complex $K$ such that $\zk$ is not a co-$H$-space.
Inequality $(1)$ can also be strict (using a construction of Katth\" an~\cite{MR3627285} in Example~\ref{EKatthan}).
However, we show the converse does hold (along with extension to larger LS-categories) for $1$-spheres and $2$-spheres, $2$-dimensional closed oriented surfaces, as well as a natural class of higher dimensional spheres, and that they can be characterised combinatorially.

\begin{proposition} 
\label{catsurface}
Let $K$ be a triangulation of a closed oriented connected surface. Then 
$$
1\leq\cuplen(\zk)=\cat(\zk)\leq 3.
$$
Moreover, letting $m:=\cat(\zk)$, we have:\\
$(i)$ $m=1$ iff $K$ has no chordless cycles (equivalently $K=\bd\Delta^3$)\\
$(ii)$ $m=2$ iff $K$ has a chordless cycle, but none with more than $3$ vertices\\
$(iii)$ $m=3$ iff $K$ has a chordless cycle with at least $4$ vertices.~$\qqed$
\end{proposition}

\begin{theorem}
\label{TM1}
If $K$ is any triangulated $d$-sphere for $d\leq 2$, 
or built up as a connected sum of joins of one or more of such spheres 
(as long as spheres in the joins are not simplex boundaries, and connected sums are over disjoint faces),
then the following are equivalent: \\
(a) $\cat(\zk)\leq k$;\\ (b) $K$ is $k$-annihilating; \\
(c) length $k+1$ cup products of positive degree elements in $H^*(\zk)$ vanish;\\
(d) there does not exist a \emph{spherical filtration} of full subcomplexes of $K$ of length more than $k$.
Moreover, $k\leq d+1$.~$\qqed$
\end{theorem}
Dually, these operations on spheres can be stated in terms of simple polyotpes, that is, the join of simplicial complexes corresponds to the product of simple polytopes, while the connected sum of simplicial complexes corresponds to taking the vertex cut of corresponding polytopes and then gluing them along the new hyperplane which is subsequently removed. 

As applications of the Theorem~\ref{TM1}, we obtain the LS category of moment-angle manifolds over important classes of simple 3-polytopes which contain fullerenes, L\"obel, Pogorelov polytopes. The moment-angle manifolds over these 3-polytopes have been of interest in various research areas, in particular there is a direct relation to hyperbolic geometry~\cite{MR3635437}.

The question of determining higher Massey products in $H^*(\zk)$ is an important one but notoriously difficult 
and equally interesting for algebraist, topologists and geometers. 
Currently, a systematic answer is known in the case of moment-angle complexes associated to one dimensional simplicial complexes, 
and only for triple Massey products of three dimensional cohomological classes (see~\cite{MR2330154, MR4181064}).
Using the Lusternik-Schnirelmann category of moment-angle complexes, 
we give the first structural results on higher Massey products in moment-angle complexes by considering $n$-Massey products of decomposable classes of arbitrary dimensions.
For instance, $k$ is precisely $3$ when $K$ is the boundary of the dual of a \emph{fullerene} $P$. 
\begin{theorem}
For fullerenes $P$, $\cat(\z{\bd P^\ast})=3$ and $\bd P^\ast$ is $3$-annihilating.
In particular, all Massey products of decomposable elements in $H^+(\z{\bd P^\ast})$ must vanish.~$\qqed$
\end{theorem}

The last theorem can be generalised to any Pogorelov polytope.
\begin{theorem}
For Pogorelov polytopes $P$, $\cat(\z{\bd P^*})=3$ and $\bd P^*$ is $3$-annihilating implying that all Massey products of decomposable elements 
in $H^+(\z{\bd P^*})$ vanish.
\end{theorem}
Applying \emph{vertex doubling operations}, the range of spheres in Theorem~\ref{TM1} can be extended.

\begin{theorem}
\label{TM2}
If $K(J)$ is the simplicial wedge of $K$ for some integer sequence $J=(j_1,\ldots,j_n)$, then $\cat(\z{K(J)})\leq \cat(\zk)$.~$\qqed$
\end{theorem}

Our methods also apply to compute tight bounds on the LS-category for other important classes of simplicial complexes, for example, those that are graphs, and those that have the neighbourly property.
 
 \begin{theorem}
If $G$ is a graph, $cat(\z{G})\leq 2$. Thus, $G$ is at least $2$-annihilating, 
and all Massey products of decomposable elements in $\Tor^+_{\Z[v_1,\ldots,v_n]}(\Z[G],\Z))$ vanish.
\end{theorem}

\begin{theorem}
\label{TM0}
If $K$ is $l$-neighbourly, then $\cat(\zk)\leq \frac{1+\dim K}{l}$ and $K$ is $\paren{\frac{1+\dim K}{l}}$-annihilating.~$\qqed$
\end{theorem}


We close the introduction by noting that many of the results in this paper  extend to polyhedral products of the form $(Cone(X),X)^K$ in place of $\zk=(D^2,S^1)^K$.

\section{Preliminary}

Recall the following concepts from~\cite{MR516214,MR1990857,MR0004108,MR0229240}. 
The \emph{geometric category} $\gcat(X)$ of a space $X$ is the smallest integer $k$ such that $X$ admits
a categorical covering $\{U_0,\ldots,U_k\}$ of $X$ with each $U_i$ contractible (in itself),
and the \emph{category} $\cat(f)$ of a map $f\colon\seqm{X}{}{Y}$ is the smallest $k$ such that $X$ admits
an open covering $\{V_0,\ldots,V_k\}$ such that $f$ restricts to a nullhomotopic map on each $V_i$.
It is easy to see that $\cat(X)=\cat(\ID\colon\seqm{X}{}{X})$, 
\begin{equation}
\label{Efcat}
\cat(f\colon\seqm{X}{}{Y})\leq\min\{\cat(X),\cat(Y)\}
\end{equation}
and $\cat(h\circ h')\leq \cat(h')$. For path-connected paracompact spaces,
$$
\cat(f\times g)\leq \cat(f)+\cat(g)
$$ 
which follows from the also well-known fact that $\cat(X\times Y)\leq \cat(X)+\cat(Y)$ together with the preceding inequalities.
Unlike $\cat()$, $\gcat()$ is not a homotopy invariant,
though one can obtain a homotopy invariant from $\gcat()$ by defining the \emph{strong category}
$$
\Cat(X)=\min\cset{\gcat(Y)}{Y\simeq X}.
$$
In fact, the strong category satisfies $\Cat(X)-1\leq \cat(X)\leq\Cat(X)\leq \gcat(X)$.
We shall let $\cuplen(X)=\Nill H^*(X)-1$ denote the length of the longest non-zero cup product of positive degree elements in $H^*(X)$. 
The main use of this is the classical lower bound 
$$
\cuplen(X)\leq \cat(X). 
$$

\subsection{Some general bounds}

We begin by giving upper bounds for the Lusternik-Schnirelmann category of some general spaces.

\begin{lemma}
\label{LOpen}
Let $A$ be a subcomplex of $X$ and $S$ an open subset of $A$. 
Then $S$ is a deformation retract of an open subset $U$ of $X$ such that $U\cap A=S$.
\end{lemma}

\begin{proof}

Let $\mc I_j$ be an index set for the $j$-cells $e_\alpha^j$ of $X-A$, $\Phi_\alpha\colon\seqm{D^j}{}{X}$ its characteristic map,
and $\phi_\alpha\colon\bd D^j\incladdr{}{D^j}\mapaddr{\Phi_\alpha}{X}$ its attaching map. 
Given a subset $B\subseteq X$, let $V_{\alpha,B}$ be the image of 
$\phi_\alpha^{-1}(B)\times[0,\frac{1}{2})\subseteq D^j\cong (\bd D^j\times [0,1])/(\bd D^j\times\{1\})$ under $\Phi_\alpha$.
Notice that $V_{\alpha,B}$ deformation retracts onto a subspace of $\phi_\alpha(\bd D^j)\cap B$,
and if $B\cap e_\alpha^j=\emptyset$, $B\cup V_{\alpha,B}$ deformation retracts onto $B$.

Construct $R_{i+1}\subseteq X$ such that $R_i\subseteq R_{i+1}$, $R_i$ is a deformation retract of $R_{i+1}$, 
and $R_i\cap X^{\vbr{i}}$ is open in the $i$-skeleton $X^{\vbr{i}}$, 
by letting $R_0=S$ and $R_{i+1}=R_i\cup \bigcup_{\alpha\in\mc I_{i+1}} V_{\alpha,S}$.
Then $U=\bigcup_{i\geq 0} R_i$ is open in $X$, deformation retracts onto $S$, and $U\cap A=S$. 

\end{proof}

\begin{lemma}
\label{LFiltration}

Given a filtration $X_0\subseteq\cdots\subseteq X_m=X$ of subcomplexes of a $CW$-complex $X$,
suppose $X_{i+1}-X_i$ is contractible in $X$ for each $i$. Then $\cat(X)\leq \cat(\incl{X_0}{}{X})+m\leq \cat(X_0)+m$.

\end{lemma}

\begin{proof}

Let $k:=\cat(\incl{X_0}{}{X})$, and $\{U_0,\ldots, U_k\}$ be a categorical cover of the inclusion \incl{X_0}{}{X}.
Note that $V_i=X_{i+1}-X_i$ is open in $X_{i+1}$ since the subcomplex $X_i$ is closed in $X_{i+1}$. 
Then iterating Lemma~\ref{LOpen}, we have open subsets $\bar U_i$ and $\bar V_i$ that deformation retract onto
each $U_i$ and $V_i$, respectively. Since the $U_i$'s and $V_i$'s cover $X$ and are contractible in $X$,
so do the $\bar U_i$'s and $\bar V_i$'s, thus they form a categorical cover of $X$.

\end{proof}

For any spaces $X$ and $Y$, and a fixed basepoint $\ast\in X$, 
we let $X\rtimes Y:=(X\times Y)/(\ast\times Y)$ denote the \emph{right half-smash} of $X$ and $Y$,
and $Y\ltimes X=(Y\times X)/(Y\times\ast)$ the \emph{left half-smash}.

\begin{lemma}
\label{LHalfSmash}
If $X$ and $Y$ are $CW$-complexes and $X$ is path-connected, then $\cat(X\rtimes Y)=\cat(X)$. 
\end{lemma}

\begin{proof}

Let  $\tilde X$ be given by attaching the interval $[0,1]$ to $X$ 
by identifying $0\in [0,1]$ with the basepoint $\ast\in X$, and fix $1\in\tilde X$ to be the basepoint.
Given $k=\cat(\tilde X)$ and $\{U_0,\ldots,U_k\}$ a categorical cover of $\tilde X$,
take the open cover $\{U_0\rtimes Y,\ldots, U_k\rtimes Y\}$ of $\tilde X\rtimes Y=(\tilde X\times Y)/(1\times Y)$
(here $U_i\rtimes Y=U_i\times Y$ if $1\notin U_i$).
Notice that the contractions of each $U_i$ in $\tilde X$ can be taken so that $1$ remains fixed if $1\in U_i$.
If $U_i$ contracts to a point $b_i$ in $\tilde X$, $U_i\rtimes Y$ deforms onto $\{b_i\}\rtimes Y$ in $\tilde X\rtimes Y$,
which in turn contracts to the basepoint in $\tilde X\rtimes Y$ by homotoping the coordinate $b_i$ to $1$.
Therefore $\cat(\tilde X\rtimes Y)\leq k$, and we have $\cat(X\rtimes Y)\leq k$ since $X\simeq\tilde X$ 
and $X\rtimes Y\simeq \tilde X\rtimes Y$. Moreover, $\cat(X\rtimes Y)\geq k$ since $X$ is a retract of $X\rtimes Y$.

\end{proof}

Let $\mc S$ be $m$ copies of the interval $[0,1]$ glued together at the endpoints $1$ in some order.
Given a collection of maps \seqm{X}{f_i}{Y_i} for $i=1,\ldots,m$, the \emph{homotopy pushout} $P$ of the maps $f_i$ 
is the $m$-fold mapping cylinder 
$$
P:=\paren{Y_1\sqcup\cdots\sqcup Y_m\sqcup (X\times\mc S)}/\sim
$$ 
under the identification $(x,t)\sim f_i(x)$ whenever $t$ is in the $i^{th}$ copy of $[0,1]$ in $\mc S$ and $t=0$.

\begin{lemma}
\label{LGT}
Fix $m\geq 2$. For $i=1,\ldots, m$, 
let $A_i$ and $C_i$ be basepointed $CW$-complexes, $B_i=\prod_{j\neq i} A_j$, and $E$ be a contractible space.
Suppose \seqm{A_i\times E}{f_i}{C_i} are nullhomotopic maps,
and $P$ is the homotopy pushout of the maps 
\seqm{A_i\times E\times B_i}{f_i\times\ID_{B_i}}{C_i\times B_i} for $i=1,\ldots, m$.
Then 
$$
\cat(P)\leq \max\{1,\cat(C_1),\ldots,\cat(C_m)\}. 
$$
\end{lemma}

\begin{proof}

We proceed by induction on $m$. Start with $m=2$. By Lemma~$7.1$ in~\cite{MR3084441}, there is a splitting 
$P\simeq (\Sigma A_1\wedge A_2)\vee (C_1\rtimes A_2)\vee(C_2\rtimes A_1)$. Thus using Lemma~\ref{LHalfSmash}, 
$$
\cat(P)=\max\{\cat(\Sigma A_1\wedge A_2),\cat(C_1\rtimes A_2),\cat(C_2\rtimes A_1)\}=\max\{1,\cat(C_1),\cat(C_2)\}.
$$
The statement holds when $m=2$.

Take $\mc B_0=\ast$, $\mc B_\ell=\prod_{j\leq \ell} A_j$, $B'_i=\prod_{j\neq i,j<m} A_j$, and $B_i$ as basepointed subspaces of $\mc B=\prod_j A_j$.
Let $P'$ be the homotopy pushout of $f_i\times\ID_{B'_i}$ for $i=1,\ldots, m-1$ (these are all maps from $E\times B_m=E\times \mc B_{m-1}$).
Suppose the lemma holds whenever $m<m'$ for some $m'>2$. Let $m:=m'$. Then $\cat(P')\leq \max\{1,\cat(C_1),\ldots,\cat(C_{m-1})\}$. 
Notice that $P$ is the homotopy pushout of $f_m\times\ID_{B_m}$ and the inclusion \seqm{A_m\times E\times B_m}{\ID_{A_m}\times g}{A_m\times P'},
where $g$ is the inclusion $W_{m-1}\subset P'$, and $W_\ell=E\times \mc B_\ell\times\{1\}$. 
We can deform $W_\ell$ into $W_{\ell-1}$ in $P'$ as follows. First deform $W_\ell$ onto $f_\ell(A_\ell\times E)\times \mc B_{\ell-1}$ by moving it
down the mapping cylinder $M=((E\times B_m\times [0,1])\sqcup (C_\ell\times B_\ell))/\sim\,$ of $P'$ and onto the base $C_\ell\times B_\ell$,
then deform it onto $\ast\times \mc B_{\ell-1}$ in $C_\ell\times B_\ell$ using the nullhomotopy of $f_\ell$. Finally, 
move $\mc B_{\ell-1}$ back up towards the top of the mapping cylinder $M$ and into $W_{\ell-1}$. Composing these deformations for $\ell=m-1,\ldots,1$
gives a contraction in $P'$ of $W_{m-1}$ to a point. Thus, $g$ is nullhomotopic, as is $f_m$. 
Since the lemma holds for the base case $m=2$, $\cat(P)=\max\{1,\cat(P'),\cat(C_m)\}\leq \max\{1,\cat(C_1),\ldots,\cat(C_m)\}$.

\end{proof}

\begin{lemma}
\label{LGluing}
Fix $m\geq 2$, and for $i=1,\ldots, m$, let $A_i$, $C_i$, $E$ be basepointed $CW$-complexes, $E$ is path-connected, and let $B_i:=\prod_{j\neq i} A_j$.
Suppose \seqm{A_i\times E}{f_i}{C_i} are maps such that the restriction $(f_i)_{|A_i\times\ast}$ of $f_i$ to $A_i\times\ast$ is nullhomotopic,
and $P$ is the homotopy pushout of the maps \seqm{A_i\times E\times B_i}{f_i\times\ID_{B_i}}{C_i\times B_i} for $i=1,\ldots, m$.

(i)
Then 
$$
\cat(P)\leq \max\{1,\cat(C_1),\ldots,\cat(C_m)\}+\Cat(E). 
$$

(ii) Moreover, if each $C_i\times B_i$ is a subcomplex of some $CW$-complex $X_i$ such that 
$X_i-C_i\times B_i$ is contractible in $X_i$, 
and $P'$ is the homotopy pushout of the maps $A_i\times E\times B_i\mapaddr{f_i\times\ID_{B_i}}{C_i\times B_i}\incladdr{}{X_i}$
for $i=1,\ldots, m$, then also
$$
\cat(P')\leq \max\{1,\cat(C_1),\ldots,\cat(C_m)\}+\Cat(E).
$$ 
\end{lemma}

\begin{proof} (i) 
Let $\mc B:=A_1\times\cdots\times A_m$, and $\mc D=\coprod_{i=1,\ldots,m}(C_i\times B_i)$,
and let $\mc S_t$ for $t<1$ be $m$ copies of the interval $[t,1]$ glued together at the endpoints $1$, 
and $\mc S'_t$ be its interior, namely, $m$ copies of $(t,1]$ glued at $1$.

Let $k:=\Cat(E)$ and take $E'\simeq E$ to be such that $k=\gcat(E')$.
Then $P$ is homotopy equivalent to the homotopy pushout $Q$ of the maps
\seqm{A_i\times E'\times B_i}{f'_i\times\ID_{B_i}}{C_i\times B_i} for $i=1,\ldots, m$, 
where $\seqm{A_i\times E'}{f'_i}{C_i}$ is the composite of $f_i$ with the homotopy equivalence
$\seqm{A_i\times E'}{\ID_{A_i}\times\simeq}{A_i\times E}$.
Since $E$ is path-connected and $\gcat()$ is unaffected by attaching an interval $[0,1]$ to a space,
we may assume that the homotopy equivalence \seqm{E'}{\simeq}{E} is basepointed for some $\ast\in E'$. 

Let $U_0,\ldots,U_k$ be an open cover of $E'$ with each $U_i$ a contractible subspace.
Take $Q_j$ to be the homotopy pushout of 
\seqm{A_i\times U_j\times B_i}{g_{i,j}\times\ID_{B_i}}{C_i\times B_i} for $i=1,\ldots, m$, 
where $g_{i,j}$ is the restriction of $f'_i$ to $A_i\times U_j$, and let
$V_j=Q_j-\mc D\,\cong\,U_j\times\mc B\times \mc S'_0$. 
Since $g_{i,j}\times\ID_{B_i}$ restricts $f'_i\times\ID_{B_i}$, 
$Q_j$ is a subspace of $Q$ and $V_j$ is open in $Q$.
Moreover, we may contract $V_j$ in $Q$ to a point as follows.
Let $\mc B_0:=\ast$ and $\mc B_\ell=\prod_{i\leq\ell} A_i\,\subseteq\,\mc B$,
and take the subspace $W_\ell=\ast\times\mc B_\ell\times\{1\}$ of $E'\times\mc B\times\mc S'_0\subset Q$.
We can deform $W_\ell$ into $W_{\ell-1}$ in $Q$, first by deforming $W_\ell$ onto 
$f'_\ell(A_\ell\times\ast)\times\mc B_{\ell-1}$ by moving it down the mapping cylinder 
$M=((E'\times \mc B\times [0,1])\sqcup (C_\ell\times B_\ell))/\sim\,$ of $Q$ and onto $C_\ell\times B_\ell$,
then deforming it onto $\ast\times \mc B_{\ell-1}$ in $C_\ell\times B_\ell$ 
using the nullhomotopy of $(f'_\ell)_{|A_\ell\times\ast}$, and finally, 
moving $\mc B_{\ell-1}$ back up towards the top of the mapping cylinder $M$ and into $E'\times \mc B\times\{1\}$.
Composing these deformations of $W_\ell$ into $W_{\ell-1}$ in $Q$ for $\ell=m,m-1,\ldots,1$,
and deforming $V_j$ onto $W_m$ using contractibility of $U_j$ and $\mc S'_0$ (onto $1$),
gives our contraction of $V_j$ in $Q$ to a point.

Assume $\ast\in U_0$. 
Since $U_0$ is contractible and $(g_{i,0})_{|A_i\times\ast}=(f'_i)_{|A_i\times\ast}$ is nullhomotopic,
$g_{i,0}$ is also nullhomotopic. Lemma~\ref{LGT} then applies to $Q_0$, namely, we have 
$$
\cat(Q_0)\leq \max\{1,\cat(C_1),\ldots,\cat(C_m)\}.
$$
Let $\mc R:=\mc S'_0-\mc S_{\frac{1}{2}}\,\cong\,\coprod_{i=1,\ldots,m}(0,\frac{1}{2})$
and $\bar{\mc R}=\mc S_0-\mc S_{\frac{1}{2}}\,\cong\,\coprod_{i=1,\ldots,m}[0,\frac{1}{2})$,
and consider the open subspace $Q'_0 = Q_0\cup(E'\times \mc B\times\mc R)$ of $Q$.
Notice $Q'_0$ deformation retracts in the weak sense onto $Q_0$ by deformation retracting the subspace of $Q_0$
$$
\paren{(E'\times \mc B\times\bar{\mc R})\,\sqcup\,\mc D\,}/\sim
$$ 
onto $\mc D$, this being done by contracting each copy of $[0,\frac{1}{2})$ in the factor $\bar{\mc R}$ to $0$, 
at the same time expanding $(U_0\times\mc B)\times \mc S_{\frac{1}{2}}$ in $Q'_0$ by expanding each copy of $[\frac{1}{2},1]$ 
in the factor $\mc S_{\frac{1}{2}}$ outwards to $[0,1]$. Then $\cat(Q'_0)=\cat(Q_0)$.
So take $k'=\max\{1,\cat(C_1),\ldots,\cat(C_m)\}$ and $\{U'_0,\ldots,U'_{k'}\}$ to be a categorical cover for $Q'_0$.
Notice that $U'_i$ is open in $Q$ since $Q'_0$ is, and $Q=\bigcup_{j=0}^n Q_j = Q'_0\cup \bigcup_{j=1}^n V_j$.
As each $V_j$ is open and contractible in $Q$, 
then $\{U'_0,\ldots,U'_{k'},V_1,\ldots,V_k\}$ is a categorical cover of $Q$. 
Therefore $\cat(P)=\cat(Q)\leq k'+k$.

(ii)
Since $C_i\times B_i$ is a subcomplex of $X_i$, $P$ is a subspace of $P'$ with 
$$
P'-P=\cdunion{i=1,\ldots,m}{}{(X_i-C_i\times B_i)},
$$ 
so $P'-P$ is open and contractible in $P'$. Notice each $V_j$ is an open (and contractible) subset of $P'$, 
while $S_j=U'_j\cap(\coprod_{i=1,\ldots,m} X_i)$ is an open subset of $\mc D$. 
By Lemma~\ref{LOpen}, there exists an open subset $R_j$ of $\coprod_{i=1,\ldots,m} X_i$ that deformation retracts onto $S_j$
such that $R_j\cap \mc D=S_j$.
Then $R'_j=R_j\sqcup (U'_j-S_j)$ is an open subset of $P'$ that deformation retracts onto $U'_j$, 
thus is contractible in $P'$. Since $P'-P$ and $V_j$ are both open in $P'$, and $(P'-P)\cap V_j=\emptyset$,
then the subspace $(P'-P)\sqcup V_j$ is contractible in $P'$. 
We can therefore take $\{R'_0,\ldots,R'_{k'},(V_1\sqcup(P'-P)),V_2,\ldots,V_k\}$ as a categorical cover for $P'$,
so $\cat(P')\leq k'+k$.
\end{proof}

\section{Moment-Angle Complexes}

Given a simplicial complex $K$ on vertex set $[n]$ and a sequence of pairs of spaces 
$$
\mc S:=((X_1,A_1),\ldots,(X_n,A_n)), 
$$ 
$A_i\subseteq X_i$, the \emph{polyhedral product} $\mc S^K$ is the subspace of $X^{\times n}$ defined by
$$
\mc S^K:=\bigcup_{\sigma\in K}Y^\sigma_1\times\cdots\times Y^\sigma_n,
$$ 
where $Y^\sigma_i=X_i$ if $i\in\sigma$, or $Y^\sigma_i=A_i$ if $i\notin\sigma$.
If the pairs $(X_i,A_i)$ are all equal to the same pair $(X,A)$, we usually write $\mc S^K$ as $(X,A)^K$.
The \emph{moment-angle complex} $\zk$ is defined as the polyhedral product $(D^2,\bd D^2)^K$,
and the \emph{real moment-angle complex} $\rzk$ is the polyhedral product $(D^1,\bd D^1)^K$.

The \emph{join} of two simplicial complexes $K$ and $L$ is the simplicial complex 
$K\ast L=\cset{\sigma\sqcup\tau}{\sigma\in K,\,\tau\in L}$,
and one has $|K\ast L|\cong |K|\ast|L|\simeq\Sigma|K|\wedge|L|$ and $\z{K\ast L}\cong\zk\times\zl$.
If $I\subseteq [n]$, $K_I=\cset{\sigma\in K}{\sigma\subseteq I}$ denotes the \emph{full subcomplex} of $K$ on vertex set $I$,
in which case $\z{K_I}$ is a retract of $\zk$.
Notice that if $K_I$ and $L_J$ are full subcomplexes of $K$ and $L$, 
then $K_I\ast L_J$ is the full subcomplex $(K\ast L)_{I\sqcup J}$ of $K\ast L$.
As a convention, we let $\z{\emptyset}:=\ast$ when $\emptyset$ is on empty vertex set.

We let $S^0$ denote both the $0$-sphere and the simplicial complex $\bd\Delta^1$ consisting of only two vertices. 
Generally, we assume our simplicial complexes (except $\emptyset$) are non-empty and have no ghost vertices,
unless stated otherwise.

\subsection{The Hochster theorem}
\label{SHochster}

When $R$ is a field or $\mb Z$, it was shown in~\cite{MR1897064,MR2255969,MR2117435,MR0441987} that there are isomorphisms 
of graded commutative algebras
\begin{equation}
\label{EHochster}
H^*(\zk;R)\cong \mathrm{Tor}_{R[v_1,\ldots,v_n]}(R[K],R) \cong \cplus{I\subseteq[n]}{}{\tilde H^*(\Sigma^{|I|+1}|K_I|;R)}.
\end{equation}
The isomorphism on the left is induced by a quasi-isomorphism of DGAs between the Koszul complex of the Stanley-Reisner ring $R[K]$
and the cellular cochain complex of $\zk$ with coefficients in $R$. 
The multiplication on the right is given by maps \seqm{H^*(K_I)\otimes H^*(K_J)}{}{H^{*+1}(K_{I\cup J})} 
that are zero when $I\cap J\neq\emptyset$, otherwise they are induced by maps
$\iota_{I,J}\wcolon\seqm{|K_{I\cup J}|}{}{|K_I\ast K_J|\cong |K_I|\ast |K_J|\simeq\Sigma |K_I|\wedge |K_J|}$
geometrically realising the canonical inclusions \incl{K_{I\cup J}}{}{K_I\ast K_J}.
One can iterate so that any length $\ell$ product
\seqm{\bigotimes^\ell_{i=1} H^*(K_{I_i})}{}{H^{*+\ell-1}(K_{I_1\cup\cdots\cup I_\ell})} is induced by the inclusion 
$$
\iota_{I_1,\ldots,I_\ell}\wcolon\incl{|K_{I_1\cup\cdots\cup I_\ell}|}{}{|K_{I_1}\ast\cdots\ast K_{I_\ell}|}
$$ 
where the $I_i$'s are mutually disjoint.

\subsection{A necessary condition}

The Hochster theorem lets us make statements about general bounds on the category $\cat(\zk)$ in terms of combinatorics and 
topology of $K$ and its full subcomplexes. Suppose $\cat(\zk)\leq \ell-1$, so cup products of length $l$ vanish in $H^+(\zk)$.
Then in light of the Hochster theorem, the inclusions $\iota_{I_1,\ldots,I_\ell}$ must induce trivial maps on cohomology. 
In fact, their suspensions must be nullhomotopic.

\begin{proposition}
If $\cat(\zk)\leq \ell-1$, then
$$
\Sigma^{m+1}\iota_{I_1,\ldots,I_\ell}\wcolon\incl{|K_{I_1\cup\cdots\cup I_\ell}|}{}{|K_{I_1}\ast\cdots\ast K_{I_\ell}|}
$$ 
is nullhomotopic for all mutually disjoint $I_1,\ldots,I_\ell\subseteq [n]$.
\end{proposition}

\begin{proof}

Let $\hzk:=\zk/\cset{(x_1,\ldots,x_n)\in \zk}{\mbox{at least one }x_i=\ast}$.
Fix $m=|I_1\cup\cdots\cup I_\ell|$, $Y=\z{K_{I_1\cup\cdots\cup I_\ell}}$, and $\hat Y=\hz{K_{I_1\cup\cdots\cup I_\ell}}$. 
Since $Y$ is a retract of $\zk$, $\cat(Y)\leq\ell-1$.
Recall from~\cite{MR516214} that a path-connected basepointed $CW$-complex such as $Y$ satisfies $\cat(Y)\leq\ell-1$ if and only if 
there is a map \seqm{Y}{\psi}{FW_\ell(Y)} such that the diagonal map \seqm{Y}{\vartriangle}{Y^{\times\ell}} 
factors up to homotopy as \seqmm{Y}{\psi}{FW_\ell(Y)}{include}{Y^{\times\ell}}.
Here $FW_\ell(Y)=\cset{(y_1,\ldots,y_\ell)\in Y^{\times \ell}}{\mbox{at least one }y_i=\ast}$ is the fat wedge.  
This implies the reduced diagonal map
$\bar\vartriangle\wcolon\seqmmm{Y}{\vartriangle}{Y^{\times\ell}}{}{Y^{\times\ell}/FW_\ell(Y)}{\cong}{Y^{\wedge\ell}}$
is nullhomotopic. Then so is
$\zeta\wcolon\seqmmm{Y}{\bar\vartriangle}{Y^{\wedge\ell}}{}{\bigwedge_j\z{K_{I_j}}}{}{\bigwedge_j\hz{K_{I_j}}}$,
where the second last map is the smash of the coordinate-wise projection maps onto each $\z{K_{I_j}}$, 
and the last map is the smash of quotient maps. 
This last nullhomotopic map $\zeta$ coincides with $\seqmm{Y}{q}{\hat Y}{\hat\iota}{\bigwedge_j\hz{K_{I_j}}}$,
where $q$ is the quotient map and $\hat\iota$ is the inclusion given simply by rearranging coordinates. 
Moreover, 
$\hat\iota$ is homeomorphic to $\Sigma^{m+1}\iota_{I_1,\ldots,I_\ell}$ and $\Sigma q$ has a right homotopy inverse
(c.f.~\cite{MR2673742}, and also the proof of Proposition~$2.5$ and pg. $23$ in~\cite{MR3658721}).
It follows that $\Sigma^{m+1}\iota_{I_1,\ldots,I_\ell}$ is nullhomotopic.

\end{proof}


\subsection{Filtrations and lower bounds}

To give combinatorial lower bounds for $\cat(\zk)$ we construct cup products in $\cat(\zk)$  using information coming from the 
combinatorics of $K$. We state this in terms of being able to construct certain filtrations of $K$ of a bounded length. Let
$$
K\sm L :=  \cset{\tau\in K}{\sigma\not\subseteq\tau\mbox{ for any }\sigma\in L}.
$$
denote the \emph{deletion} of subcomplex $L$ from $K$. Notice $K\sm L$ and $K\sm (K\sm L)$ are full subcomplexes of $K$ on complementary
vertex sets: if $K$ is on vertex set $[n]$ and $K\sm L$ is on vertex set $I\subset[n]$, then $K\sm (K\sm L)$ is on the same vertex set $[n]-I$ 
as $L$, and we can take the canonical inclusion
$$
\incl{K}{}{K_I\ast K_{[n]-I} \,=\, (K\sm L)\ast K\sm (K\sm L)}.
$$
If $L$ itself is a full subcomplex, then $L=K\sm (K\sm L)$.

\begin{lemma}
\label{LSchoenflies}
If $K$ is a triangulated $d$-sphere on vertex set $[n]$ and $L$ is a full subcomplex that is a triangulated $(d-1)$-sphere,
then $|K\sm L|\simeq|K|\sm |L|\simeq |E_1|\sqcup |E_2|$ such that $E_1$ and $E_2$ are disjoint contractible subcomplexes of $K$.

More generally, if $K$ is a triangulated closed $d$-manifold and $L$ is a full subcomplex that is a triangulated $(d-1)$-sphere 
which bounds a triangulated $d$-disk in $K$, then at least one of $E_1$ or $E_2$ is contractible.
\end{lemma}

\begin{proof}
For the case $|K|\cong S^d$, since $L$ and $K$ are finite, the embedding of $|L|$ in $|K|$ can be thickened on either side to an embedding of a thickened 
sphere. Thus by the generalised Schoenflies theorem $|L|$ bounds two distinct $d$-disks (hemispheres) $|D_1|$  and $|D_2|$ of the sphere $|K|$, and 
$|K_{[n]-I}|\cong |E_1|\sqcup |E_2|$ for full subcomplexes $E_i\subsetneq D_i$ obtained by deleting vertices from $D_i$ that are in $L$. Notice that $|D_i|$ 
deformation retracts onto $|E_i|$ as follows. Take an open collar neighbourhood $C:=[0,1)\times S^{d-1}$ of $|L|$ in $|D_i|$ small enough such that the collar 
does not intersect $|E_i|$ (i.e. is contained only in the $d$-faces that have vertices on the boundary $|L|=\bd|D_i|$), and deformation retract $|D_i|$ onto 
$|D_i|\sm |C|$ by contracting the collar onto $S^{d-1}$.
Take a $d$-face $\sigma$ in $D_i$ with vertices on $L$, and write $\sigma:=\{v_1,\ldots,v_k,w_1,\ldots,w_{d-k}\}$ so that the $v_i$'s are in $L$
and $w_i$'s are in $E_i$. There must be at least one of each type since $L$ is a full subcomplex with no $d$-faces. For any point $x\in |\sigma|$ 
that is not on the $(d-1)$-face $\{v_1,\ldots,v_k\}$, $x$ lies on the line interpolating barycentric coordinates
$$
l(t)\,:=\,(1-t)(d_1,\ldots,d_k,e_1,\ldots,e_{d-k}) + t\frac{(0,\ldots,0,e_1,\ldots,e_{d-k})}{e_1+\ldots+e_{d-k}}
$$ 
in $|\sigma|$, where $(d_1,\ldots,d_k,e_1,\ldots,e_{d-k})$ are the barycentric coordinates of $x$ in $|\sigma|$. When $x\in D_i\sm C$, these lines end at 
$t=1$ on a point in $E_i$ since at least one of $e_i>0$. Also, they vary continuously as $x$ varies, and they agree on shared boundaries of $d$-faces. 
Thus following these lines from $t=0$ to $t=1$ defines a deformation retractions of $|D_i|\sm |C|$ onto $|E_i|$, so 
$|E_i|\simeq |D_i|\sm |C|\simeq |D_i|\simeq|\ast|$ and we are done. The argument when $|K|$ is a $d$-manifold is similar.
\end{proof}

\begin{lemma}
\label{LFundClass1}
If $K$ is a triangulated $d$-sphere on vertex set $[n]$ and $L$ is a full subcomplex that is a triangulated $(d-1)$-sphere, then the inclusion
$$
\iota\wcolon\incl{|K|}{}{|(K\sm L)\ast L| \,\cong\, |K\sm L|\ast |L|}
$$
is a homotopy equivalence \seqm{S^d}{}{S^d}.

More generally, if $K$ is a triangulated closed $d$-manifold and $L$ is a full subcomplex that is a triangulated $(d-1)$-sphere which bounds a 
triangulated $d$-disk in $K$, then $\iota$ induces an isomorphism on degree $d$ cohomology mapping a generator to the fundamental class of $K$.
\end{lemma}

\begin{proof}
For the case $|K|\cong S^d$, by Lemma~\ref{LSchoenflies}, $|K\sm L|\cong |E_1|\sqcup |E_2|$ for some disjoint contractible subcomplexes 
$|E_i|$ of $|K|$. Let 
$$
h\wcolon\seqm{|K\sm L|\ast |L|}{}{S^0\ast |L|\,\cong\, S^d}
$$ 
be the join of the identity map on the right factor with the map collapsing $|E_1|$ and $|E_2|$ to $-1$ and $1$ in $S^0=\{-1,1\}$ on the right factor. 
Then $h$ is a homotopy equivalence since each $|E_i|$ is contractible. Likewise, $h\circ\iota$ is the quotient map that collapses each $|E_i|$ to a
distinct point, so it is a homotopy equivalence since the $E_i$'s are disjoint contractible subcomplexes of $K$. Therefore, $\iota$ is a homotopy equivalence.

To prove the general case, assume by Lemma~\ref{LSchoenflies} that $E_2$ is contractible. Composing $h$ with the map $q$ that collapses the bottom 
hemisphere of $S^d$ to $-1$, we see that $q\circ h\circ \iota$ is the map that collapses everything outside the interior of $d$-disk $D$ in $K$ that bounds 
$L$ to the point $-1$, and collapses the subcomplex $E_2$ in the interior of $D$ to the point $1$. But since $E_2$ is contractible, $q\circ h\circ \iota$ is 
homotopy equivalent to the map that simply collapses the everything outside the interior of $D$, which is the quotient map of the top $d$-cell of the 
$d$-manifold $K$ to a $d$-sphere. Since this induces a map on cohomology that sends the fundamental class of $S^d$ to the fundamental class of $K$, 
we are done.

\end{proof}

\begin{definition}
Given $K$ is a triangulated closed connected $d$-manifold on vertex set $[n]$, suppose there is a sequence
$I_\ell\subsetneq\cdots\subsetneq I_1=[n]$  such that the filtration of full subcomplexes
$$
K_{I_\ell}\subsetneq\cdots\subsetneq K_{I_1}:=K
$$ 
satisfies
\begin{enumerate}
\item $K_{I_i}$ is a triangulation of a $(d+1-i)$-sphere when $i\geq 2$;
\item $K_{I_2}$ bounds the triangulation of a $d$-disk in $K=K_{I_1}$
(i.e. there exists $I_2\subsetneq J\subsetneq [n]$ such that $|K_J|\cong D^d$).
\end{enumerate}
Then we say that this is a \emph{spherical filtration} of $K$ of length $\ell$. 
\end{definition}

\begin{remark} 
Condition (2) is redundant when $K$ is a triangulated $d$-sphere by the generalised Schoenflies theorem.
\end{remark}

\begin{definition}
For any triangulated closed manifold $K$, define its \emph{full filtration length} $\filt(K)$ to be the largest integer $\ell$ 
such that $K$ admits a spherical filtration of length $\ell$.
\end{definition}

\begin{proposition}
\label{PFiltCup}
If some full subcomplex $K_I\subseteq K$ is a triangulated closed manifold with a spherical filtration of length $\ell$,
then $\ell\leq\cuplen(\zk)$.
\end{proposition}
\begin{proof}

Let $K_{I_\ell}\subsetneq\cdots\subsetneq K_{I_1}:=K_I$ be a spherical filtration of $K_I$ length $\ell$. 
Let $J_{i+1}=I_i-I_{i+1}$ for $i<\ell$. By Lemma~\ref{LFundClass1}, each inclusion
$$
\iota_i\wcolon\incl{|K_{I_i}|}{}{|K_{I_{i+1}}\ast K_{J_{i+1}}|\cong |K_{I_{i+1}}|\ast |K_{J_{i+1}}|}
$$ 
is a homotopy equivalence when $i\geq 2$. Take the composite of inclusions
\begin{equation}
\label{EComposite}
|K_{I_1}|\incladdr{\iota_1}{|K_{I_2}\ast K_{J_2}|}\incladdr{\iota'_2}{|K_{I_3}\ast K_{J_3}\ast K_{J_2}|}\incladdr{\iota'_3}{\cdots}
\incladdr{\iota'_{\ell-1}}{|K_{I_\ell}\ast K_{J_\ell}\ast\cdots\ast K_{J_2}|}
\end{equation}
where the $i^{th}$ map ($i\geq 2$) in this composite  
$$
\iota'_i\wcolon
\incl{|K_{I_i}\ast K_{J_i}\ast\cdots\ast K_{J_2}|}{}{|K_{I_{i+1}}\ast K_{J_{i+1}}\ast K_{J_i}\ast\cdots\ast K_{J_2}|}
$$
is the join of the homotopy equivalence $\iota_i$ and the identity \seqm{|K_{J_i}\ast\cdots\ast K_{J_2}|}{\ID_i}{|K_{J_i}\ast\cdots\ast K_{J_2}|}.
Since each $|K_{J_j}|\simeq S^0$ for $j\geq 3$, then $\ID_i$ is homotopy equivalent to the identity \seqm{S^{i-2}}{\ID}{S^{i-2}}.
Each $\iota'_i$ (and $\iota_i$) is therefore a homotopy equivalence when $i\geq 2$, and then so is the composite in~\eqref{EComposite} of all the maps
except possibly the first. The first map $\iota_1$ on the other hand induces a non-trivial map to the fundamental class of $K_I$ by Lemma~\ref{LFundClass1}.
Therefore~\eqref{EComposite} induces a non-trivial map on cohomology mapping the fundamental class of $S^d$  to the fundamental class of $K_I$, 
so the Hochster theorem implies there is a non-trivial length $\ell$ cup product in $H^+(\z{K_{I_1}})$.
\end{proof}

\begin{corollary}
\label{CFiltCup}
If $K$ is a triangulated closed connected manifold, then 
$$
\filt(K)\leq\cuplen(\zk)\leq \cat(\zk).
$$
More generally for any $K$, since each $\z{K_I}$ is a retract of $\zk$, then 
$$
\filt(K_I)\leq\cuplen(\zk)\leq \cat(\zk)
$$
whenever $K_I$ is a triangulated closed connected manifold.
~$\qqed$
\end{corollary}

With some additional effort the statement in Lemma~\ref{LFundClass1} can be generalized so that it does away with the triangulated sphere needing to be a full 
subcomplex. In particular, Proposition~\ref{PFiltCup} still holds when we modify the definition of spherical filtration so that the second full subcomplex $K_{I_2}$ 
in the filtration is replaced with any triangulated $(d-1)$-sphere $L$ that is not necessarily a full subcomplex, as long as it satisfies the hypothesis in Lemma~\ref{LFundClass2}
below. The proof of Proposition~\ref{PFiltCup} then follows as before, this time using  Lemma~\ref{LFundClass2} in place of Lemma~\ref{LFundClass1} for the map $\iota_1$.

\begin{lemma}
\label{LFundClass2}
Suppose $K$ is a triangulated closed connected $d$-manifold, $L\subsetneq K$ is a triangulation of a $(d-1)$-sphere that bounds a triangulation of a $d$-disk $D$ in $K$ 
such that $D\sm L$ is non-empty and connected, and $K\sm L$ consists of two non-empty disconnected components $A_0$ and $A_1$. Then the inclusion
$$
\iota\wcolon\incl{|K|}{}{|(K\sm L)\ast K\sm (K\sm L)|}
$$
induces an isomorphism on degree $d$ cohomology, mapping the fundamental class of $S^d$ to the fundamental class of $K$.
\end{lemma}

\begin{proof}
Denote $A:=A_0\sqcup A_1:=K\sm L$ and $B:=K\sm A$, and let $D\subsetneq K$ be a choice of triangulated $d$-disk that is bounded by the triangulated 
$(d-1)$-sphere $L$. Notice that $D$ cannot be a subcomplex of $B$ and all vertices of $D$ not in $L$ must all be in exactly one of $A_0$ or $A_1$ since 
$D\sm L$ is non-empty and connected.

Consider the inclusion $g\colon\seqm{L}{}{B}$ restricting to the identity on vertices. Topologically, the inclusion $g\colon\seqm{|L|}{}{|B|}$ has a left 
homotopy inverse 
$$
r\wcolon\seqm{|B|}{}{|L|\cong S^{d-1}}
$$
by the following argument. Note that $B$ has no faces above dimension $d$. Also, no $d$-faces can form a cycle in the simplicial chain group $C_d(B)$. 
If they did, then this would at the same time be a cycle in $C_d(K)$ that is distinct from the cycle $\xi$ corresponding to the fundamental class of $H^*(K)$ 
consisting of all $d$-faces of $K$  (since $K\sm L$ has two non-empty disconnected components), and they could not be homologous to themselves or zero since there are no $(d+1)$-faces, meaning $H_d(K)$ has two $\mb Z$ generators, which contradicts Poincar\'e duality. 
Then $H_*(B)$ is trivial in degrees $d$ and above, implying $B$ is homotopy equivalent to at most a $(d-1)$-dimensional $CW$-complex, and so we can quotient the 
$(d-2)$-skeleton of $B$ to obtain a map \seqm{|B|}{}{\bigvee_\alpha S^{d-1}} into a possibly empty wedge of $(d-1)$-spheres that induces a surjection on 
$H_{d-1}$. Thus, all that is left to show is that $g$ induces a non-trivial map on $H_{d-1}$. In this direction, take the cycle in $C_{d-1}(L)$ consisting of all 
$(d-1)$-faces of $L$ yielding its fundamental class, which is a cycle $\gamma$ in $C_{d-1}(B)$ under the inclusion $g$. We are done if we can show $\gamma$ is not a 
boundary in $C_{d-1}(B)$. To see this, suppose conversely that $\gamma$ is a boundary of an element $a\in C_d(B)\subsetneq C_d(K)$. Since the $d$-disk $D$ is 
bounded by $L$, $\gamma$ is also boundary of the element $b\in C_d(K)$ consisting of the $d$-faces of $D$, and $a\neq b$ since $D$ is not a subcomplex of $B$. 
Then $a-b$ is a cycle in $C_d(K)$ distinct from the cycle $\xi$ (since it does not contain those $d$-faces with a vertex in one of $A_0$ or $A_1$), 
which contradicts Poincar\'e duality by the same argument as before. 

Now consider the composite
$$
f\colon\seqmmm{|K|}{\iota}{|A\ast B|\cong |A|\ast|B|}{\ID\ast r}{|A|\ast S^{d-1}}{q\ast\ID}{S^0\ast S^{d-1} \cong S^d}
$$
where the last map $q\ast\ID$ collapses the components $A_0$ and $A_1$ of $A$ to either point of $S^0$. Thus $f$ maps $|A_0|$ and $|A_1|$ to distinct 
hemispheres of $S^d\cong S^0\ast S^{d-1}$ sharing the equator $S^{d-1}\subsetneq S^d$. In turn, $f$ maps the pair $(|D|,|L|)$ to one of the hemispheres such that 
$f$ restricts on $|L|$ to $r\circ g$, mapping $|L|$ to the equator. In other words, $f$ restricts to a map of pairs \seqm{(|D|,|L|)}{f'}{(D^d,S^{d-1})}, which represents an 
element of $\pi_d(D^d,S^{d-1})$ since $(|D|,|L|)\cong (D^d,S^{d-1})$. Since the boundary map \seqm{\pi(D^d,S^{d-1})}{\bd}{\pi_{d-1}(S^{d-1})} is an isomorphism 
$\mb Z\mapsto\mb Z$ in the homotopy long exact sequence of the pair $(D^d,S^{d-1})$, and since $\bd([f'])=[r\circ g]$ and $r\circ g$ is homotopic to the identity,
then $f'$ represents the identity, meaning $f'$ is homotopic to a homeomorphism \seqm{(|D|,|L|)}{}{(D^2,S^{d-1})}. Then using the homotopy extension property, 
$f$ is homotopic to a map that restricts to such a homeomorphism, while still mapping everything outside $|D|$ to the opposing hemisphere. Now collapsing the opposing
hemisphere to a point, $f$ is homotopic to a map that collapses everything outside a $d$-disk in the $d$-manifold $K$ to a point. Namely, this is the map 
quotienting the $d$-cell of $|K|$ to a sphere that induces an isomorphism on degree $d$ cohomology between their fundamental classes. 
Since $f$ factors through $\iota$, we are done.

\end{proof}

\subsection{Skeleta and suspension on coordinates}

Let $K$ be a simplicial complex on vertex set $[n]$ and let $\sk{K}{i}$ denote the $i$-skeleton of $K$, and $\sk{K}{-1}:=\emptyset$. 
An inclusion of simplicial complexes \incl{L}{}{K} induces a canonical inclusion of $CW$-complexes \incl{\zl}{}{\zk}. 
This gives $\zsk{K}{i}$ and $\zsk{K}{-1}=(\bd D^2)^{\times n}=(S^1)^{\times n}$ as $CW$-subcomplexes of $\zk$.

\begin{lemma}[Corollary~$3.3$ in~\cite{MR3084441}]
\label{LGT2}
If $K$ is on vertex set $[n]$ with no ghost vertices, then $\zsk{K}{-1}=(\bd D^2)^{\times n}$ is contractible in $\zk$.~$\qqed$
\end{lemma} 

\begin{lemma}
\label{LSkeleta}
If $0\leq l\leq \dim K$, then $\zsk{K}{\ell}-\zsk{K}{\ell-1}$ is contractible in $\zk$.
\end{lemma}

\begin{proof}
We have a decomposition
$$
\zsk{K}{\ell}-\zsk{K}{\ell-1}=\cdunion{\sigma\in K,\,|\sigma|=\ell+1}{}{\tilde Y_1^\sigma\times\cdots\times \tilde Y_n^\sigma}
$$
where $\tilde Y_i^{\sigma}=D^2-\bd D^2$ if $i\in\sigma$ and $\tilde Y_i^{\sigma}=\bd D^2$ if $i\notin\sigma$. 
This being a disjoint union of open subspaces of $\zsk{K}{l}$, 
each of which can be deformed into $\zsk{K}{-1}$ in $\zk$
by contracting $\tilde Y_i^{\sigma}$ to a point in $\bd D^2$ whenever $i\in\sigma$. 
Thus $\zsk{K}{\ell}-\zsk{K}{\ell-1}$ can also be deformed into $\zsk{K}{-1}$.
Then $\zsk{K}{\ell}-\zsk{K}{\ell-1}$ is contractible in $\zk$ by Lemma~\ref{LGT2}.

\end{proof}

\begin{lemma}
\label{LSkeleta2}
If $-1\leq j\leq \dim K$, then 
$$
\cat(\zk)\leq \cat(\incl{\zsk{K}{j}}{}{\zk})+\dim K-j.
$$
In particular,
$$
\cat(\zk)\leq \dim K+1.
$$
and
$$
\cat(\zk)\leq \cat(\zsk{K}{j})+\dim K-j.
$$
\end{lemma}

\begin{remark}
Lemma~\ref{LSkeleta} and~\ref{LSkeleta2} can be generalised to any filtration 
$L_j\subseteq \cdots\subseteq L_k=K$ satisfying $\bd\sigma\subseteq L_i$ whenever $\sigma\in L_{i+1}$
in place of the skeletal filtration.
\end{remark}

\begin{proof}
The skeletal filtration $\sk{K}{j}\subseteq\cdots\subseteq\sk{K}{\dim K}=K$
induces a filtration of subcomplexes $\zsk{K}{j}\subseteq\cdots\subseteq\zk$, 
and  for $0\leq j\leq \dim K$, $\zsk{K}{j}-\zsk{K}{j-1}$ is contractible in $\zk$ by Lemma~\ref{LSkeleta}.
The result then follows using Lemma~\ref{LFiltration}. In particular, when $j=-1$, 
we get $\cat(\zk)\leq \dim K+1$ since $\cat(\incl{\zsk{K}{-1}}{}{\zk})=0$ by Lemma~\ref{LGT2}.
The last bound follows from equation~\ref{Efcat}.
\end{proof}

\begin{proposition}
\label{PCoordSusp}
Consider the sequences of pairs of spaces $\mc S:=((X_1,A_1),\ldots,(X_n,A_n))$ 
and $\mc T:=((\Sigma^{m_1}X_1,\Sigma^{m_1} A_1),\ldots,(\Sigma^{m_n} X_n,\Sigma^{m_n} A_n))$
for some integers $m_i$ and connected basepointed $X_i$. Then for any $K$ with no ghost vertices,
$$
\cat(\mc T^K)\leq \cat(\mc S^K).
$$
\end{proposition}

\begin{proof}
Let $K$ be on vertex set $[n]$, $k:=\cat(\mc S^K)$, and take a categorical cover $\{U_0,\ldots,U_k\}$ of $\mc S^K$.
For any open subset $V$ of $\mc S^K$, define the following open subset $V^1$ of $\mc T^K$
$$
V^1 \,:\,= \cset{((t_1,x_1),\ldots,(t_n,x_n))\in \prod_{i=1}^n \Sigma^{m_i}X_i}{(x_1,\ldots,x_n)\in V,\,t_i\in D^{m_i}}.
$$
In particular, $\mc T^K=(\mc S^K)^1$. Then $\{U^1_0,\ldots,U^1_k\}$ is an open cover of $\mc T^K$. 
Since $K$ has no ghost vertices, $A_i\subseteq X_i$, and each $X_i$ is path-connected, then $\mc S^K$ is path-connected. 
Since $\Sigma^{m_i}X_i$ is the reduced suspension of the basepointed space $X_i$, 
we have identifications $(t,\ast)\sim\ast\in\Sigma^{m_i}X_i$. 
Then we can define a contraction of $U^1_i$ in $\mc T^K$ by contracting $U_i$ in $\mc S^K$ to a point $p$ 
and homotoping $p$ to the basepoint $(\ast,\ldots,\ast)\in\mc S^K$.
Therefore, $\{U^1_0,\ldots,U^1_k\}$ is a categorical cover of $\mc T^K$. 

\end{proof}

Notice that the $(i+1)$-skeleton $\sk{(\rzk)}{i+1}$ of $\rzk$ is equal to $\rzsk{K}{i}$
although this is not true for the complex moment-angle complex $\zk$.

\begin{corollary}
\label{CRealMAC}
For any $K$ with no ghost vertices,
\begin{equation}
\label{mac}
\cat(\zk)\leq \cat(\rzk)
\end{equation}
and if $\rzk$ is not contractible and $i\geq 0$, then
\begin{equation}
\label{RmacCmac}
\cat(\zsk{K}{i})\leq \cat(\rzsk{K}{i})\leq \cat(\rzk).
\end{equation}
\end{corollary}

\begin{proof}
Inequality~\eqref{mac} and the first inequality in~\eqref{RmacCmac} follow from Proposition~\ref{PCoordSusp}.  
By the main corollary of Theorem~$1$ in~\cite{MR1933583},
the $i$-skeleton $\sk{X}{i}$ of any connected non-contractible $CW$-complex $X$
satisfies that $\cat(\sk{X}{i})\leq\cat(X)$. 
Since $\sk{(\rzk)}{i+1}=\rzsk{K}{i}$ holds for real moment-angle complexes, the last inequality follows.

\end{proof}

It is plausible that the second bound can be strengthened to $\cat(\zsk{K}{i})\leq \cat(\zk)$.
In any case, even if it is true, we will sometimes need a sharper bound.

Let $X$ and $Y$ be path-connected paracompact spaces, and $\mc U:=\{U_0,\ldots,U_k\}$ and $\mc V:=\{V_0,\ldots,V_{\ell}\}$ 
be categorical covers of $X$ and $Y$, respectively. We recall James' construction of a categorical cover $\mc W:=\{W_0,\ldots,W_{k+\ell}\}$
of $X\times Y$ from the covers $\mc U$ and $\mc V$ (see~\cite{MR516214}, page $333$). 

Let $\{\pi_j\}_{j\in\{0,\dots,k\}}$ be a partition of unity subordinate to the cover $\mc U$. 
For any subset $S\subseteq \{0,\dots,k\}$, define 
$$
W_{\mc U}(S)\,:=\,\cset{x\in X}{\pi_j(x)>\pi_i(x)\mbox{ for any }j\in S\mbox{ and }i\notin S},
$$ 
and for any point $p\in X$, let 
$$
S_{\mc U}(p)\,:=\,\cset{j\in \{0,\dots,k\}}{\pi_j(p)>0}.
$$
Since the context is clear, let $W(S):=W_{\mc U}(S)$ and $S(p)=S_{\mc U}(p)$.
Then $W(S)$ is an open subset of $X$. 
Given $x\in X$, $x\in W(S)$ where $S=\cset{i}{\,\pi_i(x)=\max\{\pi_1(x),\ldots,\pi_k(x)\}\,}$), we have $X=\bigcup_{S\subseteq\{0,\dots,k\}}W(S)$.
Moreover, $W(S')\cap W(S)=\emptyset$ when $S\nsubseteq S'$ and $S'\nsubseteq S$, in particular, when $|S|=|S'|$ and $S\neq S'$, and $W(S)\subseteq U_j$
whenever $j\in S$. 
Therefore $W(S)$ is contractible in $X$. Then so is the disjoint union of open sets
\begin{equation}
\label{EUi}
U'_i\,:=\,\cdunionmulti{S=S(p)\mbox{\tiny{ for some} }p\in X}{|S|=i+1}{}{W(S)}.
\end{equation}
Since $W(S)=\emptyset$ when $S\neq S(p)$ for every $p\in X$,
the set $\{U'_0,\ldots,U'_k\}$ forms a categorical cover of $X$.
We obtain a categorical cover $\{V'_0,\ldots,V'_\ell\}$ of $Y$ from $\mc V$ by an analogous construction.

Now let $\bar U_i:=U'_{k-i}\cup\cdots\cup U'_k$ and $\bar V_j:=V'_{\ell-j}\cup\cdots\cup V'_\ell$,
and for $-1\leq s\leq k+\ell$, let $C_{-1}:=\emptyset$ and
$$
C_s:=\cunionmulti{i+j=s}{i\leq k,\,j\leq\ell}{}{\bar U_i\times\bar V_j}.
$$
Take $W_s:=C_s-C_{s-1}$. Notice that
\begin{equation}
\label{EWs}
W_s=\cdunionmulti{i+j=s}{i\leq k,\,j\leq\ell}{}{U'_i\times V'_j}.
\end{equation}
This defines a categorical cover $\mc W$ of $X\times Y$.

Given subcomplexes $B\subseteq Y$ and $A\subseteq X$, consider the polyhedral product 
$$
\mc X^{S^0}:=X\times B \cup_{A\times B} A\times Y
$$
over the sequence $\mc X:=((X,A),(Y,B))$, where $S^0$ is considered as a simplicial complex of two disjoint points.

\begin{lemma}
\label{LPoly}
If $X-A$ is contractible in $X$ and $Y-B$ is contractible in $Y$, then
$$
\cat(\mc X^{S^0})\leq \cat(A)+\cat(B)+1.
$$
\end{lemma} 

\begin{proof}
Suppose we have categorical covers $\{R_1,\ldots,R_k\}$ and $\{S_1,\ldots,S_\ell\}$
of $A$, and $B$. By Lemma~\ref{LOpen}, we have open subsets $U_i\subseteq X$ and $V_i\subseteq Y$
such that $U_i$ and $V_i$ deformation retract onto $R_i$ and $S_i$ respectively, 
and $U_i\cap A=R_i$ and $V_i\cap B=S_i$ for $i\geq 1$. Then taking $U_0:=X-A$ and $V_0:=Y-B$, 
$\mc U:=\{U_0,\ldots,U_k\}$ and $\mc V:=\{V_0,\ldots,V_\ell\}$ are categorical covers of $X$ and $Y$.

Notice that
$$
X\times Y-U_0\times V_0=\mc X^{S^0},
$$
and since $R_i=U_i\cap A=U_i-U_0$ and $S_j=V_j\cap B=V_j-V_0$ for $i,j\geq 1$, 
\begin{equation}
\label{EDiff}
D_{i,j}\,:=\,U_i\times V_j-U_0\times V_0\,=\,(R_i\times V_j)\cup_{R_i\times S_j}(U_i\times S_j). 
\end{equation}
Notice that $D_{i,j}$ is contractible in $\mc X^{S^0}$ by deformation retracting the factor
$U_i$ onto $R_i$ and $V_j$ onto $S_j$,
then contracting $R_i\times S_j$ in $A\times B$.

Take the categorical cover $\mc W:=\{W_0,\ldots,W_{k+\ell}\}$ of $X\times Y$ constructed from 
$\mc U$ and $\mc V$ as above. By~\eqref{EWs},
\begin{align*}
W_s-U_0\times V_0 &=\cdunionmulti{i+j=s}{i\leq k,\,j\leq\ell}{}{(U'_i\times V'_j-U_0\times V_0)},
\end{align*}
and by~\eqref{EUi},
\begin{align*}
U'_i\times V'_j-U_0\times V_0=\cdunion
{\indexsize{\begin{matrix}
\scriptstyle S=S_{\mc U}(p)\mbox{\tiny{ for some} }p\in X \cr 
\scriptstyle T=S_{\mc V}(q)\mbox{\tiny{ for some} }q\in Y \cr 
\scriptstyle |S|=i+1,\,|T|=j+1
\end{matrix}}}
{}{(W_{\mc U}(S)\times W_{\mc V}(T)-U_0\times V_0)}.
\end{align*}
These are disjoint unions of open subsets of $\mc X^{S^0}$.
Since $W_{\mc U}(S)$ is contained in some $U_{i'}$ and
$W_{\mc V}(T)$ is contained in some $V_{j'}$, it follows that
$(W_{\mc U}(S)\times W_{\mc V}(T)-U_0\times V_0)$ is contained in $D_{i',j'}$, 
so it is contractible in $\mc X^{S^0}$. 
Therefore, 
so are the disjoint unions $U'_i\times V'_j-U_0\times V_0$ and $W_s-U_0\times V_0$. 
Moreover, since $W_{k+\ell}=U'_k\times V'_\ell$,
and $U'_k=W_{\mc U}(\{0,\ldots,k\})$ and $V'_\ell=W_{\mc V}(\{0,\ldots,\ell\})$ 
are contained in $U_{i'}$ and $V_{j'}$ respectively for each $i'\in\{0,\ldots,k\}$ and $j'\in\{0,\ldots,\ell\}$,
$W_{k+\ell}-U_0\times V_0=\emptyset$. Then
$$
\{(W_0-U_0\times V_0),\ldots,(W_{k+\ell-1}-U_0\times V_0)\}
$$
is a categorical cover of $\mc X^{S^0}$.

\end{proof}

\begin{corollary}
\label{CJoinSkeleton}
Let $K$ and $L$ be simplicial complexes with $d:=\dim K$ and $d':=\dim L$. Then
$$
\cat(\zsk{(K\ast L)}{d+d'})\leq \cat(\zsk{K}{d-1})+\cat(\zsk{L}{d'-1})+1.
$$
\end{corollary}

\begin{proof}
Recall that $\dim K\ast L=d+d'+1$. Notice that $$\sk{(K\ast L)}{d+d'}=(K\ast\sk{L}{d'-1})\cup_{(\sk{K}{d-1}\ast\sk{L}{d'-1})} (\sk{K}{d-1}\ast L),$$ 
$$\z{K\ast L}=\zk\times\zl, \text{ so}$$
\begin{align*}
\zsk{(K\ast L)}{d+d'} &= (\z{K\ast\sk{L}{d'-1}}) \cup_{\z{(\sk{K}{d-1}\ast\sk{L}{d'-1})}} (\z{\sk{K}{d-1}\ast L})\\
&= (\zk\times\zsk{L}{d'-1})\cup_{\zsk{K}{d-1}\times\zsk{L}{d'-1}} (\zsk{K}{d-1}\times\zl),
\end{align*}
and $\zk-\zsk{K}{d-1}$ and $\zl-\zsk{L}{d'-1}$ are contractible in $\zk$ and $\zl$ by Lemma~\ref{LSkeleta}.
The result follows by Lemma~\ref{LPoly}.

\end{proof}

\subsection{Missing face complexes}

Take $K$ on vertex set $[n]$. We fix the basepoint in the unreduced suspension $\Sigma|K|=(|K|\times [0,1]) /\sim$ 
to be the tip of the double cone corresponding to $1$ under the identifications $(x,0)\sim 0$ and $(x,1)\sim 1$.
Let $MF(K):=\cset{\sigma\subseteq[n]}{\sigma\notin K,\,\bd\sigma\subseteq K}$ be the collection of minimal missing faces of $K$.
With this we can define a large class of category $1$ moment-angle complexes that include those over chordal graphs
(this is a somewhat more flexible alternative to the \emph{directed missing face complexes} defined in~\cite{MR3507473}).
By filtering through skeleta as in the previous section, tight upper bounds can sometimes be obtained when the 
category of the moment angle complex over the $1$-skeleton is known to be small.

\begin{definition}
A simplicial complex $K$ on vertex set $[n]$ is called a \emph{homology missing face complex} (or \emph{HMF-complex}) if for each non-empty 
$I\subseteq [n]$, $K_I$ is a simplex or there exists a subcollection $\mc C_I\subseteq MF(K_I)$ such that the wedge sum of suspended inclusions
$$
\gamma_I\wcolon\seqm{\cvee{\sigma\in \mc C_I}{}{\Sigma|\bd\sigma|}}{}{\Sigma|K_I|}
$$
induces an isomorphism on homology. Consequently $\gamma_I$
 is a homotopy equivalence since it is a map between suspensions.
\end{definition}

\begin{bigremark}
Given $H_*(K_I)$ is torsion-free, since each $\Sigma|\bd\sigma|$ is a sphere, 
one needs only to find $\gamma_I$ that induces surjection on homology in order for $K$ to be an $HMF$-complex.
\end{bigremark}

\begin{proposition}
\label{PExtractible}
If $K$ is an $HMF$-complex, then $\zk$ is homotopy equivalent to a wedge of spheres or is contractible.
Therefore $\cat(\zk)\leq 1$ and $\Cat(\zk)\leq 1$.
\end{proposition}

\begin{proof}
For each $I\subseteq [n]$, either $K_I$ is a simplex, boundary of a simplex, or else for each $\sigma\in\mc C_I$, 
we can pick an $i_\sigma\in I$ such that $\bd\sigma\subseteq K_{I-\{i_\sigma\}}$, 
so each inclusion \seqm{|\bd\sigma|}{}{|K_I|} factors through inclusions 
\seqmm{|\bd\sigma|}{}{|K_{I-\{i_\sigma\}}|}{}{|K_I|}. Take the composite
$$
f\wcolon\seqmmm{\Sigma|K_I|}{\gamma^{-1}_I}{\cvee{\sigma\in\mc C_I}{}{\Sigma|\bd\sigma|}}{}
{\cvee{i\in I}{}{\Sigma|K_{I-\{i_\sigma\}}|}}{}{\Sigma|K_I|}
$$
where $\gamma^{-1}_I$ is a homotopy inverse of $\gamma_I$,
the second last map includes the summand $\Sigma|\bd\sigma|$ into the summand $\Sigma|K_{I-\{i_\sigma\}}|$,
and the last map is the standard inclusion on each summand. 
Since the composite of the last two maps is $\gamma_I$, $f$ is a homotopy equivalence.
Then $K$ is an \emph{extractible complex} as defined in~\cite{2013arXiv1306.6221I}. 
Therefore $\zk$ is homotopy equivalent to a wedge of spheres or contractible by Corollary~$3.3$ therein.
\end{proof}

\subsection{Gluings and connected sums}
Now we look at the effect on category of moment angle complexes when two simplicial complexes are glued along a full subcomplex, 
or along a simplex with interior then deleted (i.e. a connected sum).
If $L$ and $K$ are simplicial complexes and $C$ is a full subcomplex common to both $L$ and $K$, then we obtain a new simplicial complex
$L\cup_C K$ by gluing $L$ and $K$ along $C$. One can always glue along simplices since they are always full subcomplexes.
When $C=\emptyset$, $L\cup_C K$ is just the disjoint union $L\sqcup K$. 

Given $\sigma\in K$, define the \emph{deletion} of the face $\sigma$ from $K$ to be the simplicial complex
given by 
$$
K\sm\sigma:=\cset{\tau\in K}{\sigma\not\subseteq\tau}.
$$
If $\sigma$ is a common face of $L$ and $K$, define the \emph{connected sum} $L\#_\sigma K$ to be the simplicial complex 
$(L\sm\sigma)\cup_{\bd\sigma}(K\sm\sigma)$.
In other words, $L\#_\sigma K$ is obtained by deleting $\sigma$ from $L$ and $K$ and gluing along the boundary $\bd\sigma$.
As a convention, we let $\z{\emptyset}:=\ast$ when $\emptyset$ is on empty vertex set.

\begin{proposition}
\label{PGluing2}
If $C$ is a (possibly empty) full subcomplex common to $K_1,\ldots,K_m$, then 
$$
\cat(\z{K_1\cup_C\cdots\cup_C K_m})\leq \max\{1,\cat(\z{K_1}),\ldots,\cat(\z{K_m})\}+\Cat(\z{C}).
$$
Moreover, if each $K_i$ is the $(d_i-1)$-skeleton of some $d_i$ dimensional simplicial complex $\bar K_i$, 
and $C$ is also a full subcomplex of each $\bar K_i$, then 
$$
\cat(\z{\bar K_1\cup_C\cdots\cup_C \bar K_m})\leq \max\{1,\cat(\z{K_1}),\ldots,\cat(\z{K_m})\}+\Cat(\z{C}).
$$
\end{proposition}

\begin{proof}

Let $K_i$ be on vertex set $[n_i]$, and $C$ has $\ell$ vertices. 
If $C$ is on vertex set $[n_i]$, possibly with ghost vertices, 
the inclusion \incl{C}{}{K_i} induces a coordinate-wise inclusion \incl{(\bd D^2)^{\times n_i-\ell}\times\z{C}}{f_i}{\z{K_i}}.
By Lemma~\ref{LGT2}, $f_i$ is nullhomotopic when restricted to $(\bd D^2)^{\times n_i-\ell}\times\ast$.
Let $N_i:=\Sigma_{j\neq i}n_j$. Note $\z{K_1\cup_C\cdots\cup_C K_m}$ is the pushout of
\incl{(\bd D^2)^{\times n_i-\ell}\times\z{C}\times(\bd D^2)^{\times N_i-\ell}}{f_i\times\ID}{\z{K_i}\times(\bd D^2)^{\times N_i-\ell}}
for $i=1,\ldots,m$. 
Since each of these maps are inclusions of subcomplexes, $\z{K_1\cup_C\cdots\cup_C K_m}$ is homotopy equivalent to the homotopy pushout $P$ 
of these maps. The first inequality therefore follows from the first part of Lemma~\ref{LGluing}.

By Lemma~\ref{LSkeleta}, $\z{\bar K_i}-\z{K_i}$ is contractible in $\z{\bar K_i}$,  
so the second equality follows from the second part of Lemma~\ref{LGluing}. 

\end{proof}

\begin{example}
\label{EGluing}
In particular, when $C$ is a simplex 
$\cat(\z{L\cup_C K})\leq \max\{1,\cat(\zl),\cat(\zk)\}$ and
$\cat(\z{\bar L\cup_C \bar K})\leq \max\{1,\cat(\zl),\cat(\zk)\}$ 
since $\z{C}$ is contractible.
These also hold when $C$ is the empty simplex and $\z{C}=\ast$
in which case $\bar L\cup_C \bar K=\bar L\sqcup \bar K$ and $L\cup_C K=L\sqcup K$.
When $C$ is the boundary of a simplex,   
$\cat(\z{L\cup_C K})\leq \max\{1,\cat(\zl),\cat(\zk)\}+1$
since $\z{C}$ here is a sphere.
\end{example}

The bound in Proposition~\ref{PGluing2} is not always optimal, sometimes far from it. 
If $K$ and $\Delta^{n-1}$ are on vertex set $[n]$ and $L$ is formed by gluing $\Delta^{n-1}$ and $\{n+1\}\ast K$ along $K$,
then $\zl$ is a $co$-$H$-space by~\cite{2013arXiv1306.6221I} so $\cat(\zl)=1$.
In fact, it is not difficult to directly show that $\zl\simeq\Sigma^2\zk$. 
On the other hand, 
Proposition~\ref{PGluing2} gives $\cat(\zl)\leq \max\{1,\cat(\zk)\}+\Cat(\zk)$ since $\z{\{n+1\}\ast K}\cong D^2\times\zk\simeq\zk$ 
and $\z{\Delta^n}$ is contractible.

\begin{corollary}
\label{CConnectedSum}
Suppose
$$
K=L_1\#_{\sigma_1} L_2\#_{\sigma_2}\cdots\#_{\sigma_{k-1}} L_k  
$$
where $\dim L_i=d$, $\sigma_i$ is a $d$-face common to $L_i$ and $L_{i+1}$, 
and $\sigma_i\cap\sigma_j=\emptyset$ when $i\neq j$. Then
$$
\cat(\zk) \leq \max\{1,\cat(\zsk{L_1}{d-1}),\ldots,\cat(\zsk{L_k}{d-1})\}+1.
$$
\end{corollary}

\begin{proof}
Take the disjoint unions 
$$
C:=\bd\sigma_1\sqcup\cdots\sqcup\bd\sigma_{k-1}
$$ 
$$
K_1:=\csqdunion{1\leq 2i+1\leq k-1}{}{L_{2i+1}}
$$
$$
K_2:=\csqdunion{2\leq 2i\leq k-1}{}{L_{2i}}
$$
and take the iterated face deletions $K'_1:=K_1\sm(\sigma_1\sqcup\cdots\sqcup\sigma_{k-1})$ and
$K'_2:=K_2\sm(\sigma_1\sqcup\cdots\sqcup\sigma_{k-1})$. 
Then $C$ is a full subcomplex common to both $K'_1$ and $K'_2$, 
and to both $\sk{K'_1}{d-1}$ and $\sk{K'_2}{d-1}$.
Moreover, $\sk{K'_1}{d-1}=\sk{K_1}{d-1}$ and $\sk{K'_2}{d-1}=\sk{K_2}{d-1}$,
and $K=K'_1\cup_C K'_2$, so by the second part of Proposition~\ref{PGluing2},
$$
\cat(\z{K})\leq \max\{1,\cat(\zsk{K_1}{d-1}),\cat(\zsk{K_2}{d-1})\}+\Cat(\z{C}).
$$
It is clear that $C$ is an $HMF$-complex, so $\z{C}$ is homotopy equivalent to a wedge of spheres 
and $\Cat(\z{C})=1$. Alternatively, this follows from Theorem~$10.1$ in~\cite{MR2321037}.
Moreover, we can think of $\zsk{K_1}{d-1}$ as being built up iteratively by gluing 
$\bigsqcup_{1\leq 1\leq j}{}{\zsk{L_{2i+1}}{d-1}}$ and $\zsk{L_{2j+3}}{d-1}$ along the empty simplex,
so iterating the first inequality in Example~\ref{EGluing},
$$
\cat(\zsk{K_1}{d-1})\leq \max\{\,1,\,\cat(\zsk{L_1}{d-1}),\,\cat(\zsk{L_3}{d-1}),\,\ldots\,\}.
$$ 
Likewise,
$
\cat(\zsk{K_2}{d-1})\leq \max\{\,1,\,\cat(\zsk{L_2}{d-1}),\,\cat(\zsk{L_4}{d-1}),\,\ldots\,\}.
$ 
The inequality in the lemma follows.
\end{proof}

\section{Some Specific Applications: Triangulated Surfaces and Spheres}

We now apply the general tools developed in the previous sections to compute exact values for the $LS$-category of 
moment angle complexes over triangulated closed oriented surfaces and certain classes of triangulated higher dimensional spheres that are dual
to boundaries of polytopes built up iteratively via vertex cuts and products.

\subsection{Triangulated spheres}

Let $\mc C_0:=\{S^0\}$, $\mc C_1$, and $\mc C_2$ consist of all triangulated $0$,$1$, and $2$-spheres, 
and for $d\geq 3$, let $\mc C_d$ be the class of triangulated $d$-spheres defined by $K\in \mc C_d$ if  
\begin{itemize}
\item[(1)] $K=L_1\ast\cdots\ast L_k$ for some $L_i\in\mc C_{d_i}$, $d_i\leq 2$, and $d_1+\cdots+d_k=d-k+1$;  
\item[(2)] $K=K_1\#_{\sigma_1}\cdots\#_{\sigma_{\ell-1}}K_\ell$ where $\sigma_i$ is a $d$-face common to $K_i$ and $K_{i+1}$ 
with $\sigma_i\cap\sigma_j=\emptyset$ when $i\neq j$, and each $K_i=L_{1,i}\ast\cdots\ast L_{k_i,i}$ 
is of the form (1) such that each $L_{j,i}$ is not the boundary of a simplex.
\end{itemize}

Recall that the join $L\ast L'$ is the simplicial complex $\cset{\sigma\sqcup\sigma'}{\sigma\in L,\,\sigma'\in L'}$, 
and the connected sum $K\#_{\sigma} K'$ is given topologically by gluing triangulations $K$ and $K'$ of $S^d$ 
along a common $d$-face $\sigma$, and deleting its interior.

\begin{bigremark}
If $L$ and $L'$ are boundaries $\bd P^*$ and $\bd P'^*$ of the duals of simple polytopes $P$ and $P'$,
then $L\ast L'$ is the boundary of $(P\times P')^*$, while $L\#_{\sigma} L'$ is the boundary of dual $Q^*$, 
where $Q$ is obtained by taking the vertex cut at the vertices of $P$ and $P'$ that are dual to $\sigma$, 
gluing along the new hyperplane and removing it after gluing. 
\end{bigremark}

Our goal in the next few subsections will be to show the following.

\begin{theorem}
\label{TMain}
If $K$ on vertex set $[n]=\{1,\ldots,n\}$ is any triangulated $d$-sphere for $d=0,1,2$, 
or $K\in \mc C_d$ for $d\geq 3$, then the following are equivalent.
\begin{itemize}
\item[(1)] $K$ is $m$-annihilating over $\Z$;
\item[(2)] $\Nill(\mathrm{Tor}_{\Z[v_1,\ldots,v_n]}(\Z[K],\Z))\leq m+1$ (equivalently $\cuplen(\zk)\leq m$);
\item[(3)] for any filtration of full subcomplexes 
$$
\bd\Delta^{d+2-\ell}=K_{I_\ell}\subsetneq K_{I_{\ell-1}}\subsetneq\cdots\subsetneq K_{I_1}=K
$$ 
such that $|K_{I_i}|\cong S^{d+1-i}$, we have $\ell\leq m$; 
\item[(4)] $\cat(\zk)\leq m$.
\end{itemize}
Moreover, $1\leq m\leq d+1$; that is, $K$ satisfies any of the above for some $m$ which cannot be greater than $d+1$.~$\qqed$ 
\end{theorem}

\subsection{Filtration length on spheres}
The simplest example is $\filt(\bd\Delta^{d+1})=1$. 
Generally, there are the following bounds with respect to joins, connected sums, and cup product length.

\begin{lemma}
\label{LFiltCSum}
If $K$ and $L$ are both triangulations of $S^d$, and $\sigma$ is a $d$-face common to $K$ and $L$, 
then 
$$
\filt(K\#_\sigma L)\geq \max\{2,\filt(K),\filt(L)\}.
$$
\end{lemma}

\begin{proof}

A full subcomplex $N_{\mc I}$ of $K\#_\sigma L$ satisfying $|N_{\mc I}|\cong S^k$ for some $k<d$ must either be a full subcomplex
of exactly one of $K$ or $L$, or else $N_{\mc I}=\bd\sigma$, otherwise $N_{\mc I}$ would have a $(k-1)$-face contained in three $k$-faces.
Moreover, $K\#_\sigma L$ always has the length $2$ spherical filtration $\bd\sigma\subsetneq K\#_\sigma L$.
The lemma follows immediately. 
\end{proof}

\begin{lemma}
\label{LFiltJoin}
If $K$ and $L$ are any triangulated spheres, then 
$$
\filt(K\ast L)\geq \filt(K)+\filt(L).
$$
\end{lemma}

\begin{proof}

Let $K$ and $L$ be on vertex sets $[n]$ and $[m]$.
Let $d:=\dim K$, $d':=\dim L'$, $\ell:=\filt(K)$, $\ell':=\filt(L)$, and take 
$K_{I_\ell}\subsetneq\cdots\subsetneq K_{I_1}=K$ and
$L_{J_{\ell'}}\subsetneq\cdots\subsetneq L_{J_1}=L$ 
to be spherical filtrations of $K$ and $L$.

Since $|K_{I_i}\ast L_{J_j}|\cong |K_{I_i}|\ast|L_{J_j}|\cong S^{d+1-i}\ast S^{d'+1-j}\cong S^{d+d'-i-j+3}$,
and  $|K_{I_i}\ast L_{J_j}|$ is a full subcomplex of $|K_{I_{i'}}\ast L_{J_{j'}}|$ when $i\leq i'$ and $j\leq j'$,
we have a length $\ell+\ell'$ spherical filtration of $K\ast L$
$$
(K_{I_\ell}\ast L_{J_{\ell'}})\subsetneq\cdots
\subsetneq (K_{I_\ell}\ast L_{J_2})
\subsetneq (K_{I_\ell}\ast L_{J_1})
\subsetneq (K_{I_{\ell-1}}\ast L_{J_1})
\subsetneq\cdots\subsetneq (K_{I_1}\ast L_{J_1})=K\ast L.
$$
Therefore $\filt(K\ast L)\geq\ell+\ell'$. 
\end{proof}

\subsection{Triangulated $d$-spheres for $d=0,1,2$}
The only triangulated $0$-sphere is $S^0$, and the only triangulated $1$-sphere with $n\geq 3$ vertices is the $n$-gon,
both of which have an $LS$-category that is easy to characterize. For higher dimensional spheres the following definition 
will be useful.
\begin{definition}
We will say that $C$ is a \emph{chordless cycle} in $K$ with $m\geq 3$ vertices if $C$ is a full subcomplex $K_I$ of $K$, 
and $C$ is an $m$-gon for some $m\geq 3$. When $K$ is a graph, this is the same as $C$ being an induced cycle of $K$.
A simplicial complex $K$ is said to be \emph{chordal} if it contains no chordless cycles with $4$ or more vertices.
\end{definition}

\begin{lemma} 
\label{LCup1}
If $K$ is a triangulation of $S^1$ on vertex set $[n]$, then
\begin{itemize}
\item[(i)] $\filt(K)\geq 2$ whenever $K$ has at least $4$ vertices.
\end{itemize}
If $K$ is a triangulation of $S^2$, then 
\begin{itemize}
\item[(ii)] $\filt(K)\geq 2$ whenever $K$ has a chordless cycle;
\item[(iii)] $\filt(K)\geq 3$ whenever $K$ has a chordless cycle with at least $4$ vertices.
\end{itemize}
\end{lemma}

\begin{proof}
Let $|K|\cong S^1$. If $K$ has at least $n\geq 4$ vertices, 
then $K$ being an $n$-gon means we can take $I'\subset [n]$, $|I'|=2$, 
such that $|K_{I'}|=S^0$. Then $S^0\subsetneq K$ is a length $2$ spherical filtration of $K$.

Let $|K|\cong S^2$, and $C:=K_I$ be a chordless cycle for some $I\subset [n]$. 
We have $|K_I|\cong S^1$ and $K_I=\bd\Delta^2$ when $K_I$ has $3$ vertices, 
in which case $K_I\subsetneq K$ is a length $2$ spherical filtration. 
Otherwise, when $K_I$ has at least $4$ vertices, $\filt(K)\geq 3$ since 
there is a spherical filtration $S^0\subsetneq K_I\subsetneq K$ by part~$(i)$.
\end{proof}

\begin{proposition} 
\label{Pd012}
Let $K$ be a triangulated $d$-sphere, $d=0,1,2$. Then 
$$
1\leq\filt(K)=\cuplen(\zk)=\cat(\zk)\leq d+1.
$$
In particular, letting $m:=\cat(\zk)$, when $d=1$ we have:\\
$(i)$ $m=1$ iff $K=\bd\Delta^2$\\
$(ii)$ $m=2$ iff $K$ has at least $4$ vertices.

When $d=2$ we have:\\
$(i)$ $m=1$ iff $K$ has no chordless cycles of any length (equivalently $K=\bd\Delta^3$)\\
$(ii)$ $m=2$ iff $K$ has a chordless cycle, but none with more than $3$ vertices\\
$(iii)$ $m=3$ iff $K$ has a chordless cycle with at least $4$ vertices.

\end{proposition}

\begin{proof} 

The case $d=0$ is immediate since $S^0$ is the only triangulated $0$-sphere, and $\z{S^0}\cong S^3$. 
By Lemma~\ref{LSkeleta2} and Corollary~\ref{CFiltCup}
$$
1\leq\filt(K)\leq\cuplen(\zk)\leq\cat(\zk)\leq d+1,
$$ 
so it remains to show that $\cat(\zk)\leq\filt(K)$. 

Fix $d=1$. Using Lemma~\ref{LCup1}, if $K$ has at least $4$ vertices, 
then $\filt(K)\geq 2$, so $\cat(\zk)\leq\filt(K)$ since $\cat(\zk)\leq d+1=2$. 
Otherwise, if $K$ has $3$-vertices, then $K=\bd\Delta^2$ and $\zk\cong S^5$, so $\filt(K)=\cat(\zk)=1$. 

Now fix $d=2$. We break the argument up into three cases: $K$ has no chordless cycles; $K$ has only
chordless cycles of length no more than $3$; and $K$ has at least one chordless cycle of length $4$ or more.
To start, by Lemma~\ref{LCup1}, $\filt(K)\geq 3$ whenever $K$ has a chordless cycle with at least $4$ vertices, 
so $\cat(\zk)\leq\filt(K)$ since $\cat(\zk)\leq d+1=3$. 
On the other hand, suppose $K$ has chordless cycles, but none with more than $3$ vertices. 
Then $\filt(K)=2$ by Lemma~\ref{LCup1}, and the $1$-skeleton $\sk{K}{1}$ is a chordal graph. 
The chordal property is closed under vertex deletion (taking full subcomplexes).
Moreover, recall from~\cite{MR0186421} that chordal graphs have a \emph{total elimination ordering}, that is,
they can be built up one vertex $v$ at a time in some order such that at each step the neighbours of $v$ form a clique. 
Inducting on this ordering, one sees that chordal graphs are $HMF$-complexes,
therefore $\cat(\zsk{K}{1})\leq 1$ by Proposition~\ref{PExtractible} (this also follows the main result in~\cite{2015arXiv150608970I}).
Using Lemma~\ref{LSkeleta2}, we have $\cat(\zk)\leq \cat(\zsk{K}{1})+1\leq 2=\filt(K)$.
Otherwise, $K=\bd\Delta^3$ when $K$ has no chordless cycles at all, so we have $\zk\cong S^7$ and $\filt(K)=\cat(\zk)=1$. 

\end{proof}

\begin{lemma}
\label{Ld012}
Let $K$ be a triangulated $d$-sphere, $d=1,2$. Then 
$$
\cat(\zsk{K}{d-1})\leq\max\{1,\cat(\zk)-1\}=\max\{1,\filt(K)-1\}.
$$
\end{lemma}

\begin{proof}
The last equality $\filt(K)=\cat(\zk)$ is from the previous proposition. Let $m:=\cat(\zk)$. 
Using Proposition~\ref{Pd012}, the statement simplifies to $\cat(\zsk{K}{0})\leq 1$ when $d=1$, or $d=2$ and $m=1$.
This is true since $\zsk{K}{0}$ has the homotopy type of a wedge of spheres by~\cite{MR2138475}, 
or by Proposition~\ref{PExtractible} as $\sk{K}{0}$ is a collection of disjoint points.
Fix $d:=2$. As in the proof of Proposition~\ref{Pd012}, $\cat(\zsk{K}{1})\leq 1<m$ when $m=2$, since $\filt(K)=2$. 
This last inequality also holds when $m=3$ since $\cat(\zsk{K}{1})\leq \dim\sk{K}{1}+1=2$ is always true by Lemma~\ref{LSkeleta2}. 
\end{proof}

\subsection{Triangulated $d$-spheres for $d\geq 3$}

\begin{proposition}
\label{Pdinfinity}
Suppose $K\in\mc C_d$, $d\geq 0$. Then 
$$
1\leq\filt(K)=\cuplen(\zk)=\cat(\zk)\leq d+1.
$$
\end{proposition}

\begin{proof}
By Lemmas~\ref{LSkeleta2} and Corollary~\ref{CFiltCup},
$$
1\leq\filt(K)\leq\cuplen(\zk)\leq\cat(\zk)\leq d+1.
$$ 
It remains to show that $\cat(\zk)\leq\filt(K)$. The $d=0,1,2$ case is Proposition~\ref{Pd012}.

Suppose $K=L_1\ast\cdots\ast L_k\in \mc C_d$ for some $L_i\in\mc C_{d_i}$, $d_i\leq 2$, and $d_1+\cdots+d_k=d-k-1$.
Then $\zk=\z{L_1}\times\cdots\times\z{L_k}$, and so using Proposition~\ref{Pd012} and iterating Lemma~\ref{LFiltJoin},
$$
\cat(\zk)\leq \csum{i=1,\ldots,k}{}{\cat(\z{L_i})} = \csum{i=1,\ldots,k}{}{\filt(L_i)} \leq \filt(L_1\ast\cdots\ast L_k)=\filt(K).
$$ 
Moreover, iterating Corollary~\ref{CJoinSkeleton}, and using Lemma~\ref{Ld012}, when each $L_i\neq S^0$
$$
\cat(\zsk{K}{d-1})\leq \csum{i=1,\ldots,k}{}{\cat(\zsk{L_i}{d_i-1})}+k-1\leq\csum{i=1,\ldots,k}{}{\max\{1,\filt(L_i)-1\}}+k-1
$$
and when each $L_i\neq\bd\Delta^{d_i+1}$, we have $\filt(L_i)>1$, therefore 
\begin{equation}
\label{ESk}
\cat(\zsk{K}{d-1})\leq \paren{\csum{i=1,\ldots,k}{}{\filt(L_i)}}-1\leq \filt(L_1\ast\cdots\ast L_k)-1=\filt(K)-1
\end{equation}
the second inequality by iterating Lemma~\ref{LFiltJoin}.

Suppose $K=K_1\#_{\sigma_1}\cdots\#_{\sigma_{\ell-1}}K_\ell$ where each $K_i=L_{1,i}\ast\cdots\ast L_{k_i,i}$ 
is a join of the above form such that each $L_{j,i}$ is not the boundary of a simplex,
and $\sigma_i$ is a $d$-face common to $K_i$ and $K_{i+1}$ with $\sigma_i\cap\sigma_j=\emptyset$ when $i\neq j$.
By Corollary~\ref{CConnectedSum} and inequality~\eqref{ESk}, we have 
\begin{align*}
\cat(\zk) & \leq \max\{1,\cat(\zsk{K_1}{d-1}),\ldots,\cat(\zsk{K_\ell}{d-1})\}+1\\
& \leq \max\{1,\filt(K_1)-1,\ldots,\filt(K_\ell)-1\}+1\\
& = \max\{2,\filt(K_1),\ldots,\filt(K_\ell)\}\\
& = \filt(K_1\#_{\sigma_1}\cdots\#_{\sigma_{\ell-1}}K_\ell)\\
& = \filt(K)
\end{align*}
where the second last inequality follows from iterating Lemma~\ref{LFiltCSum}.

\end{proof}

\subsection{Proof of Theorem~\ref{TMain}}

The following is an immediate consequence of Theorem~$4.4$ in~\cite{MR1644063} and the fact that 
\emph{category weight}, $cwgt$ as defined there is bounded below by $1$, and linearly below with respect to cup products.

\begin{theorem}[Rudyak~\cite{MR1644063}]
If $\cat(X)\leq m$, then
\begin{itemize}
\item[(1)] $\cuplen(X)\leq m$;
\item[(2)] Massey products $\vbr{v_1,\ldots,v_k}$ vanish in $H^*(X)$ 
whenever $v_i=a_1\cdots a_{m_i}$ and $v_j=b_1\cdots b_{m_j}$, and $m_i+m_j>m$, for some odd $i$ and even $j$ 
and $a_s,b_t\in H^+(X)$.~$\qqed$
\end{itemize}
\end{theorem}

Then by our remarks in Section~\ref{SHochster}, and by definition $\cuplen(X)=\Nill H^*(X)-1$, we have a condition for $K$ to be 
$m$-annihilating.

\begin{proposition}
\label{PRudyak}
If $\cat(\zk)\leq m$, then $K$ is $m$-annihilating.~$\qqed$ 
\end{proposition}

Now Theorem~\ref{TMain} follows from Propositions~\ref{Pd012} and~\ref{Pdinfinity}, 
and the fact that, $\cuplen(\zk)\leq m$ when $K$ is $m$-annihilating.

\subsection{Triangulated surfaces}

We extend the characterization of the $LS$-category of $2$-spheres to closed connected $2$-dimensional oriented surfaces.
As before, the algebraic invariant that distinquishes between them is cup product length, though the combinatorial characerisation diverges. 
The main difference is that spherical filtration length does not play a part for surfaces as it did for spheres since the full subcomplex
condition is a bit too strict here. Also, the different values for
the $LS$-category do not pigeonhole nicely in terms of length of chordless cycles. Note the case of $LS$-category $1$ has already been
characterized in~\cite{MR3770008}.

\begin{proposition} 
\label{catsurface}
Let $K$ be a triangulation of a closed oriented connected surface. Then 
$$
1\leq\cuplen(\zk)=\cat(\zk)\leq 3.
$$
Moreover, letting $m:=\cat(\zk)$, we have:\\
$(i)$ $m=1$ iff $K$ has no chordless cycles (equivalently $K=\bd\Delta^3$)\\
$(ii)$ $m=2$ iff $K$ has a chordless cycle, but none with more than $3$ vertices\\
$(iii)$ $m=3$ iff $K$ has a chordless cycle with at least $4$ vertices.

\end{proposition}

\begin{proof} 

For any vertex $v$, let $D_v$ be the fan of $2$-faces that contain $v$, and $\bd D_v$ be its  boundary, which consists of all vertices adjacent to $v$. 
To see that $D_v$ is a triangulation of a $2$-disk, notice that a small neighbourhood of $v$ in $|K|$ is homeomorphic to a $2$-disk since $|K|$ is a closed 
$2$-manifold, so the interior of $|D_v|$ is an open $2$-disk. At the same time the boundary $\bd D_v$ has to be a triangulated circle, otherwise we would 
have a pair of distinct edges between the same pair of vertices, namely, vertex $v$ and a vertex $w$ in $\bd D_v$, which contradicts $K$ being a triangulation. 

Now we get to the heart of the proof.
As in the proof of Proposition~\ref{Pd012}, we break the argument up into three cases: (1) $K$ has no chordless cycles of any length; 
(2) $K$ has chordless cycles of length no more than $3$; and (3) $K$ has at least one chordless cycle of length $4$ or more.

The first difficulty extending to surfaces beyond $2$-spheres is in case (3). Unlike the situation for $2$-spheres, we cannot use a chordless cycle 
$C$ of length $\geq 4$ to obtain a length $3$ spherical filtration since $C$ might not always bound a $2$-disk in $K$. 
We search for filtrations derived from $C$ instead. Taking $v$ to be a vertex on a chordless cycle $C$ of length 
$\geq 4$, since the vertices $w_0$ and $w_1$ that are adjacent to $v$ in $C$ are in $\bd D_v$, we have a length $3$ filtration of $K$ in terms of
triangulated spheres
$$
E:=\{w_0,w_1\}\subsetneq \bd D_v\subsetneq K
$$
where $|E|\cong S^0$ , $|\bd D_v|\cong S^1$, and $|\bd D_v|$ bounds the $2$-disk $|D_v|$. Also, $E$ is a full subcomplex of $K$ since there 
are no edges between $w_0$ and $w_1$ in $C$ since they are both adjacent to $v$ in $C$, and $C$ is chordless with at least $4$ vertices. So there are no edges between them in $K$ as well since $C$ is a full subcomplex of $K$ by definition of chordless cycle.
This then is almost a spherical filtration of length $3$; the only thing missing is that $\bd D_v$ might not be a full subcomplex of $K$. 
However, $\bd D_v$ satisfies the hypothesis of Lemma~\ref{LFundClass2}. In particular, $K\sm \bd D_v$ has two disconnected components:
one containing $v$, and the other one containing another vertex $w_3$ in $C$ besides $v$, $w_0$, or $w_1$ which must be disconnected 
from $v$ in $K$ since only $w_1$ and $w_2$ are adjacent to $v$ in the full subcomplex $C$. Then by the remarks preceding 
Lemma~\ref{LFundClass2}, $3\leq\cuplen(\zk)$. Namely we have a length $3$ cup product induced by the composite of inclusions
$$
\seqmm{K}{}{\bd D_v\ast (K\sm\bd D_v)}{}{E\ast (\bd D_v\sm E)\ast (K\sm\bd D_v)}.
$$
Since $\cuplen(\zk)\leq \cat(\zk)\leq 3$ by Lemma~\ref{LSkeleta2}, we are done.

The difficulty in case (2) is similar. We have $\cat(\zk)\leq \cat(\zsk{K}{1})+1\leq 2$ by the same argument as in Proposition~\ref{Pd012},
but again we cannot choose anly length $3$ chordless cycle to form a length $2$ spherical filtration. Instead, as in case (3), we look at 
$D_v$ for some vertex $v$ in a chordless cycle $C$ of length $3$. Assume $K$ is not a triangulated $2$-sphere since that is addressed in Proposition~\ref{Pd012}.
Note $\bd D_v$ must have a chordless cycle. If it did not then $\bd D_v$ would bound a triangulated $2$-disk in $D$ on the same vertex set as $\bd D_v$,
and then $|K|$ would have to be a $2$-sphere since $|D|\cup |D_v|$ is a $2$-sphere. Also, since $\bd D_v$ must have chordless cycles, then $\bd D_v$ cannot 
have every vertex in $K$ beside $v$. Otherwise, embedding $|K|$ in $\mb R^3$, we would have a path from $v$ through the interior of $|K|$
and outside $|K|$ through this chordless cycle, contradicting that $|K|$ is a closed surface. Then $\bd D_v$ satisfies the hypothesis of Lemma~\ref{LFundClass2},
and the filtration $\bd D_v\subsetneq K$ gives a length $2$ cup product similarly as case (3). Thus $2\leq\cuplen(\zk)\leq \cat(\zk)$ and we are done.

Case (1) is similar as for spheres. For any $v$, $\bd D_v$ must have at least $3$ vertices. In case (1) it cannot have more than $3$ as follows.
Suppose it had $4$ or more vertices. Since we assume $K$ has no chordless cycles, then the cycle $\bd D_v$ must not be a full subcomplex. 
Namely, there must be an edge $\{a,b\}$ that is not in $\bd D_v$ for two non-adjacent vertices $a,b\in \bd D_v$. But then $(a,b,v)$ is a length 
$3$ chordless cycle since $\sigma:=\{a,b,v\}$ cannot be a $2$-face of $K$. If it was a face in $K$ then $K$ would not be a surface since $\{b,v\}$ 
would be adjacent to three $2$-faces: $\sigma$ and the other two in the triangulated $2$-disk $D_v$ distinct from $\sigma$ since $\sigma\notin D_v$.
Now since $\bd D_v$ has $3$ vertices, and $\bd D_v$ cannot be a length $3$ chordless cycle, that is, cannot be a full subcomplex, it must be 
the boundary of a $2$-face $\delta$, so $\delta$ and $D_v$ form the boundary of a $3$-simplex in $K$, implying $K$ itself must be this boundary.
\end{proof}

\section{Further Applications}

\subsection{Fullerenes and Pogorelov polytopes}
A \emph{fullerene} $P$ is a simple $3$-polytope all of whose $2$-faces are pentagons and hexagons. 
These are mathematical idealisations of physical fullerenes - 
spherical molecules of carbon such that each carbon atom belongs to three carbon rings, and each carbon ring is either a pentagon or hexagon.

The authors in~\cite{2015arXiv151103624F} have shown that the cohomology ring of moment-angle complexes is a complete combinatorial invariant of fullerenes,
while Buchstaber and Erokhovets~\cite{MR3706860,MR3486778} show that the finer details of their cohomology encode many interesting properties of fullerenes.
For example, if $P^\ast$ is the dual of $P$, then the bigraded Betti numbers of $\z{\bd P^\ast}$ count the number $k$-\emph{belts} in $P$. 
Here, a $k$-belt of a simple polytope such as $P$ is a sequence of $2$-faces $(F_1,\ldots,F_k)$ such that
$F_k\cap F_1\neq\emptyset$, $F_i\cap F_{i+1}\neq\emptyset$ for $1\leq i\leq k-1$, and all other intersections are empty.
Notice that the $k$-belts of $P$ correspond to full subcomplexes of $\bd P^\ast$ that are $k$-gons.  
But since fullerenes can have no $3$-belts~\cite{MR3706860,MR3486778}, $\bd P^\ast$ must only have $n$-gons as full subcomplexes for $n\geq 4$.
Moreover, since $\bd P^\ast$ is a triangulated $2$-sphere that is not a boundary of the $2$-simplex, it must have at least one such $n$-gon as a full subcomplex. 
Thus, $\filt(\bd P^\ast)=3$, and by Theorem~\ref{TMain}, we can determine certain Massey products in $H^*(\z{\bd P^\ast})$.
\begin{theorem}
For fullerenes $P$, $\cat(\z{\bd P^\ast})=3$ and $\bd P^\ast$ is $3$-annihilating.
In particular, all Massey products consisting of decomposable elements in $H^+(\z{\bd P^\ast})$ must vanish.~$\qqed$
\end{theorem}

Pogorelov polytopes~\cite{MR0208458} can be seen as a generalisation of fullerenes. A Pogorelov polytope is a combinatorial simple 3-polytope realisable in the 
Lobachevsky (hyperbolic) space as a bounded right-angled polytope. These polytopes are exactly simple 3-polytopes with cyclically 5-edge connected graphs. 
A Pogorelov polytope has no 3 and 4-gons and may have any prescribed numbers of $k$-gons, $k \geq 7$. Any simple polytope with only 5-, 6 and at most 
one 7-gon is Pogorelov. 

Since Pogorelov polytopes $P$ have no 3 and 4-belts of facets, $\bd P^*$ must only have $n$-gons as full subcomplexes for $n\geq 5$. Now following the 
same proof as in the case of fullerenes, we have the following statement.
\begin{theorem}
For Pogorelov polytopes $P$, $\cat(\z{\bd P^\ast})=3$ and $\bd P^\ast$ is $3$-annihilating.
In particular, all Massey products consisting of decomposable elements in $H^+(\z{\bd P^\ast})$ must vanish.~$\qqed$
\end{theorem}


\subsection{Neighbourly complexes}

For any finite simply-connected $CW$-complex $X$, let 
$$
\hd(X):=\max\br{\,\max\cset{i}{\tilde H^i(X)\otimes\mb Q\neq 0},\,\max\cset{i}{\mathrm{Torsion}(\tilde H^{i-1}(X))\neq 0}\,}
$$ 
and 
$$
\hc(X):=\min\cset{i}{\tilde H^{i+1}(X)\neq 0}.
$$
These coincide with the dimension and connectivity of $X$ up to homotopy equivalence.
It is well known (c.f.~\cite{MR516214}) that $X$ satisfies 
\begin{equation}
\label{EOld}
\cat(X)\leq \frac{\hd(X)}{\hc(X)+1}.
\end{equation} 

A version of the Hochster formula also holds for real moment-angle complexes, namely,
\begin{equation}
H^*(\rzk)\cong\bigoplus_{I\subseteq[n]}\tilde H^*(\Sigma|K_I|). 
\end{equation} 
Thus, 
$$
\hd(\rzk)=1+\max\cset{\,\hd(|K_I|)}{I\subseteq[n]\,}\,\leq\, 1+\dim K
$$
and
$$
\hc(\rzk)=1+\min\cset{\,\hc(|K_I|)}{I\subseteq[n]\,},
$$
and using the inequality $\cat(\zk)\leq\cat(\rzk)$ from Corollary~\ref{CRealMAC}, 
\begin{proposition}
\label{PImproved}
It holds that 
$$\cat(\zk)\leq \frac{\hd(\rzk)}{\hc(\rzk)}.$$ \qed
\end{proposition} 
Comparing the Hochster formula for $H^*(\rzk)$ to the Hochster formula for $H^*(\zk)$ in Section~\ref{SHochster}, 
one sees that the inequality $\frac{\hd(\rzk)}{\hc(\rzk)}\leq \frac{\hd(\zk)}{\hc(\zk)}$ usually holds, 
with the disparity between these two often being very large. In such case, the bound in Proposition~\ref{PImproved}
is an improvement over what one would get by applying~\eqref{EOld} directly to $X=\zk$.

Consider, for instance, the case of $k$-neighbourly complexes. 
A simplicial complex $K$ on vertex set $[n]$ is said to be $k$-\emph{neighbourly} if every subset of $k$ or less vertices in $[n]$ is a face of $K$.
In this case $H_i(K_I)=0$ for $i\leq k-2$ and each $I\subseteq[n]$, so $\hc(\rzk)\geq k-1$. Therefore
we have the following result.

\begin{theorem}
If $K$ is $k$-neighbourly, $\cat(\zk)\leq \frac{1+\dim K}{k}$. In particular, $K$ is $\paren{\frac{1+\dim K}{k}}$-annihilating.~$\qqed$
\end{theorem}

\begin{example}
Suppose $K$ is $\ceil{\frac{n}{2}}$-neighbourly. 
If $K\neq\Delta^{n-1}$, then $\dim K\leq n-2$.
Thus, $\cat(\zk)\leq 1$ ($\zk$ is a $co$-$H$-space) and $K$ is $1$-annihilating.
\end{example}

\subsection{Simplicial wedges}

We recall the \emph{simplicial wedge} construction defined in~\cite{MR593648,MR3426378}. 
Let $K$ be a simplicial complex on vertex set $\{v_1,...,v_n\}$, and for any face $\sigma\in K$, 
define the \emph{link} of $\sigma$ in $K$ the subcomplex of $K$ given by
$$
\link_K(\sigma)=\cset{\tau\in K}{\tau\cap\sigma=\emptyset,\,\tau\cup\sigma\in K}.
$$
By \emph{doubling} a vertex $v_i$ in $K$, we obtain a new simplicial complex $K(v_i)$ on vertex set 
$$
\{v_1,\ldots,v_{i-1},v_{i1},v_{i2},v_{i+1},\ldots,v_n\}
$$ 
defined by
$$
K(v_i) = (v_{i1},v_{i2})\ast \link_K(v_i)\cup_{\{v_{i1},v_{i2}\}\ast\link_K(v_i)} \{v_{i1},v_{i2}\}\ast K\sm\{v_i\},
$$
where $(v_{i1},v_{i2})$ is the $1$-simplex with vertices $\{v_{i1},v_{i2}\}$.
One can of course iterate this construction by reapplying the doubling operation to successive complexes,
and the order of vertices on which this is done is irrelevant. 
To this end, taking any sequence $J=(j_1,\ldots,j_n)$ of non-negative integers, 
let $u_j$ be the $j^{th}$ vertex in the sequence $v_1,v_{12},\ldots,v_{1j_1},v_2,\ldots,v_n,v_{n2},\ldots,v_{nj_n}$
and $N=j_1+\cdots+j_n$, and define 
$$
K(J)=K_N
$$ 
where $K_{j+1}=K_j(u_{j+1})$ and $K_0=K$.
In algebraic terms, the Stanley-Reisner ideal of $K(J)$ is obtained from the Stanley-Reisner ideal of $K$ by replacing each vertex $v_i$ 
by $v_{i1},v_{i2},\ldots,v_{ij_i}$ in each monomial. 
This construction arises in combinatorics (see~\cite{MR593648}) and has the important property that if $K$ is the boundary of the dual of $d$-polytope, 
then $K(J)$ is the boundary of the dual of a $(d+N)$-polytope.

\begin{theorem}
\label{TSW} For any $J$, $\cat(\z{K(J)})\leq \cat(\zk)$.
\end{theorem}

\begin{proof}
Let $(D^J,S^J)$ be the sequence of pairs $((D^{2j_1+2},S^{2j_1+1}),\ldots,(D^{2j_n+2},S^{2j_n+1}))$.
By~\cite{MR3426378,MR3084441}, there is a homeomorphism
$$
\z{K(J)}\,=\,(D^2,S^1)^{K(J)}\,\cong\,(D^J,S^J)^K,
$$
and by Proposition~\ref{PCoordSusp}, $\cat((D^J,S^J)^K)\leq \cat((D^2,S^1)^K)=\cat(\zk)$.
\end{proof}

This result becomes algebraically useful when a good bound on $\cat(\zk)$ is known.
For instance, there are many examples of complexes $K$ for which $\cat(\zk)=1$, 
duals of sequential Cohen Macaulay and shellable complexes, and chordal flag complexes to name a few~\cite{MR3461047,2013arXiv1306.6221I}.
In each of these examples $\cat(\z{K(J)})\leq 1$, so $K(J)$ is Golod. 
Generally, $K(J)$ is at least $(1+\dim K)$-annihilating since $\cat(\zk)\leq 1+\dim K$, 
even though $\dim K(J)-\dim K$ can be arbitrarily large.

Notice $K(J)$ is a triangulation of a $(d+N)$-sphere whenever $K$ is a triangulation of a $d$-sphere.
Combining Theorem~\ref{TSW} and Proposition~\ref{PRudyak},
the range of spheres for which Theorem~\ref{TMain} holds generalises as follows.

\begin{corollary}
Let $K$ be be any triangulated $d$-sphere for $d=0,1,2$, or $K\in \mc C_d$ when $d\geq 3$,
and let $m:=\filt(K)$ (equivalently $m=\cuplen(\zk))$. Then $\cat(\z{K(J)})\leq m$ and $K(J)$ is $m$-annihilating.~$\qqed$
\end{corollary}

\subsection{Examples of complexes for which the cup product length does not determine the category}
There are two examples of simplicial complexes that clearly stand out as they are counterexamples to a conjecture and a theorem on the homotopy 
theoretical structures of moment-angle complexes associated with Golod simplicial complexes. Using $m$-annihilating property, we shall prove that in 
both cases the cup product length is strictly less than the category of these moment-angle complexes.

\begin{example}
\label{EKatthan}
For some time it was thought that for any monomial ideal $\mathfrak a$ in a polynomial ring $S=\Bbbk [x_1,\ldots, x_n]$ over some field $\Bbbk$, 
the ring $R=S/\mathfrak a$ is Golod if and only if the product in the Koszul homology $H_*(K_R)$ of $R$ is trivial~\cite{MR2344344}. 
Katth\" an~\cite{MR3627285} constructed a counterexample to this statement. Let $\Bbbk$ be a field, $S = \Bbbk[x_1,x_2,y_1,y_2,z]$ and let 
$\mathfrak a \subset S$ be the ideal with the following generators:
\begin{align}
	\ga{a} &:= x_1x_2^2 & \ga{ab} &:= x_1x_2y_1y_2  & \ga{\abSc} &:= x_1y_1z\\
	\ga{b} &:= y_1y_2^2 & \ga{bc} &:= y_2^2z^2 		& \ga{\bcSa} &:= x_2^2y_2^2z\\
	\ga{c} &:= z^3 	 & \ga{ca} &:= x_2^2z^2.
\end{align}
Then the product in $H_*(K_{S/\mathfrak a}) = \Tor^S_*(S/\mathfrak a, \Bbbk)$ is trivial, but $S/\mathfrak a$ is not Golod.
More precisely, there is a nonzero ternary Massey product. 

The polarisation of this ideal  is the Stanley-Reisner ideal of some simplicial complex $K$ of dimension $5$. By taking its $4$-skeleton, one obtains an 
example of a $4$-dimensional simplicial complex $L$, such that $\Bbbk [L]$ is not Golod but has a trivial product in its Koszul homology. 
We shall show that
\[
1=\cuplen(\mathcal Z_L)< \cat(\mathcal Z_L).
\]
As the cup product in $H^*(\mathcal Z_L)$ is trivial, $\cuplen(\mathcal Z_L)=1$. If $\cat(\mathcal Z_L)=1$, then by Proposition~\ref{PRudyak} 
 the moment-angle complex $\mathcal Z_L$ would be Golod. Therefore, $\cat(\mathcal Z_L)> \cuplen(\mathcal Z_L)$.
\end{example}
	
\begin{example}
In~\cite{MR3658721} the question whether $\zk$ is a co-$H$-space for $K$ Golod was studied and the authors showed that these two notions, 
one coming from algebra and another from topology, are very closely related. Recently, Iriye and Yano~\cite{MR3649878} described a Golod simplicial 
complex $K$ for which $\zk$ is not a co-$H$-space. It is straightforward to see that for the simplicial complex $K$,
\[
1=\cuplen(\mathcal Z_K)< \cat(\mathcal Z_K)
\]
as $\cat(\zk)\neq 1$ since $K$ is not Golod.
\end{example}

\bibliographystyle{amsplain}
\bibliography{mybibliography}

\providecommand{\bysame}{\leavevmode\hbox to3em{\hrulefill}\thinspace}
\providecommand{\MR}{\relax\ifhmode\unskip\space\fi MR }
\providecommand{\MRhref}[2]{%
  \href{http://www.ams.org/mathscinet-getitem?mr=#1}{#2}
}
\providecommand{\href}[2]{#2}
\begin{thebibliography}{10}

\bibitem{MR644015}
David~J. Anick, \emph{A counterexample to a conjecture of {S}erre}, Ann. of
  Math. (2) \textbf{115} (1982), no.~1, 1--33. \MR{644015}

\bibitem{MR2673742}
A.~Bahri, M.~Bendersky, F.~R. Cohen, and S.~Gitler, \emph{The polyhedral
  product functor: a method of decomposition for moment-angle complexes,
  arrangements and related spaces}, Adv. Math. \textbf{225} (2010), no.~3,
  1634--1668. \MR{2673742 (2012b:13053)}

\bibitem{MR3426378}
\bysame, \emph{Operations on polyhedral products and a new topological
  construction of infinite families of toric manifolds}, Homology Homotopy
  Appl. \textbf{17} (2015), no.~2, 137--160. \MR{3426378}

\bibitem{MR2117435}
I.~V. Baskakov, V.~M. Bukhshtaber, and T.~E. Panov, \emph{Algebras of cellular
  cochains, and torus actions}, Uspekhi Mat. Nauk \textbf{59} (2004),
  no.~3(357), 159--160. \MR{2117435 (2006b:57045)}

\bibitem{MR3658721}
Piotr Beben and Jelena Grbi\'c, \emph{Configuration spaces and polyhedral
  products}, Adv. Math. \textbf{314} (2017), 378--425. \MR{3658721}

\bibitem{MR3633129}
\bysame, \emph{{$\frac{n}{3}$}-neighbourly moment-angle complexes and their
  unstable splittings}, Bol. Soc. Mat. Mex. (3) \textbf{23} (2017), no.~1,
  141--152. \MR{3633129}

\bibitem{BerglundRational}
Alexander Berglund, \emph{Homotopy invariants of {D}avis-{J}anuszkiewicz spaces
  and moment-angle complexes}.

\bibitem{MR2344344}
Alexander Berglund and Michael J{\"o}llenbeck, \emph{On the {G}olod property of
  {S}tanley-{R}eisner rings}, J. Algebra \textbf{315} (2007), no.~1, 249--273.
  \MR{2344344 (2008f:13036)}

\bibitem{MR2285318}
Fr{\'e}d{\'e}ric Bosio and Laurent Meersseman, \emph{Real quadrics in {$\mathbf
  C^n$}, complex manifolds and convex polytopes}, Acta Math. \textbf{197}
  (2006), no.~1, 53--127. \MR{2285318 (2007j:32037)}

\bibitem{MR3486778}
V.~M. Buchstaber and N.~Yu. Erokhovets, \emph{Truncations of simple polytopes
  and applications}, Proc. Steklov Inst. Math. \textbf{289} (2015), no.~1,
  104--133. \MR{3486778}

\bibitem{MR3706860}
\bysame, \emph{Constructions of families of three-dimensional polytopes,
  characteristic patches of fullerenes, and {P}ogorelov polytopes}, Izv. Ross.
  Akad. Nauk Ser. Mat. \textbf{81} (2017), no.~5, 15--91. \MR{3706860}

\bibitem{MR1897064}
Victor~M. Buchstaber and Taras~E. Panov, \emph{Torus actions and their
  applications in topology and combinatorics}, University Lecture Series,
  vol.~24, American Mathematical Society, Providence, RI, 2002. \MR{1897064
  (2003e:57039)}

\bibitem{MR3363157}
\bysame, \emph{Toric topology}, Mathematical Surveys and Monographs, vol. 204,
  American Mathematical Society, Providence, RI, 2015. \MR{3363157}

\bibitem{MR3635437}
V.~M. Bukhshtaber, N.~Yu. Erokhovets, M.~Masuda, T.~E. Panov, and S.~Pak,
  \emph{Cohomological rigidity of manifolds defined by 3-dimensional
  polytopes}, Uspekhi Mat. Nauk \textbf{72} (2017), no.~2(434), 3--66.
  \MR{3635437}

\bibitem{MR1990857}
Octav Cornea, Gregory Lupton, John Oprea, and Daniel Tanr{\'e},
  \emph{Lusternik-{S}chnirelmann category}, Mathematical Surveys and
  Monographs, vol. 103, American Mathematical Society, Providence, RI, 2003.
  \MR{1990857 (2004e:55001)}

\bibitem{MR1104531}
Michael~W. Davis and Tadeusz Januszkiewicz, \emph{Convex polytopes, {C}oxeter
  orbifolds and torus actions}, Duke Math. J. \textbf{62} (1991), no.~2,
  417--451. \MR{1104531 (92i:52012)}

\bibitem{MR0382702}
Pierre Deligne, Phillip Griffiths, John Morgan, and Dennis Sullivan, \emph{Real
  homotopy theory of {K}\"{a}hler manifolds}, Invent. Math. \textbf{29} (1975),
  no.~3, 245--274. \MR{0382702}

\bibitem{MR2330154}
Graham Denham and Alexander~I. Suciu, \emph{Moment-angle complexes, monomial
  ideals and {M}assey products}, Pure Appl. Math. Q. \textbf{3} (2007), no.~1,
  Special Issue: In honor of Robert D. MacPherson. Part 3, 25--60. \MR{2330154
  (2008g:55028)}

\bibitem{2015arXiv151103624F}
F.~{Fan}, J.~{Ma}, and X.~{Wang}, \emph{{$B$-Rigidity of flag $2$-spheres
  without $4$-belt}}, ArXiv e-prints (2015).

\bibitem{MR1933583}
Yves Felix, Steve Halperin, and Jean-Claude Thomas,
  \emph{Lusternik-{S}chnirelmann category of skeleta}, Topology Appl.
  \textbf{125} (2002), no.~2, 357--361. \MR{1933583 (2004c:55025)}

\bibitem{MR2457428}
Yves F{\'e}lix and Daniel Tanr{\'e}, \emph{Rational homotopy of the polyhedral
  product functor}, Proc. Amer. Math. Soc. \textbf{137} (2009), no.~3,
  891--898. \MR{2457428 (2009i:55011)}

\bibitem{MR0004108}
Ralph~H. Fox, \emph{On the {L}usternik-{S}chnirelmann category}, Ann. of Math.
  (2) \textbf{42} (1941), 333--370. \MR{0004108 (2,320f)}

\bibitem{MR2255969}
M.~Franz, \emph{The integral cohomology of toric manifolds}, Tr. Mat. Inst.
  Steklova \textbf{252} (2006), no.~Geom. Topol., Diskret. Geom. i Teor.
  Mnozh., 61--70. \MR{2255969 (2007f:14050)}

\bibitem{MR0186421}
D.~R. Fulkerson and O.~A. Gross, \emph{Incidence matrices and interval graphs},
  Pacific J. Math. \textbf{15} (1965), 835--855. \MR{0186421 (32 \#3881)}

\bibitem{MR0229240}
T.~Ganea, \emph{Lusternik-{S}chnirelmann category and strong category},
  Illinois J. Math. \textbf{11} (1967), 417--427. \MR{0229240 (37 \#4814)}

\bibitem{MR0231375}
W.~J. Gilbert, \emph{Some examples for weak category and conilpotency},
  Illinois J. Math. \textbf{12} (1968), 421--432. \MR{0231375}

\bibitem{MR3073929}
Samuel Gitler and Santiago L{\'o}pez~de Medrano, \emph{Intersections of
  quadrics, moment-angle manifolds and connected sums}, Geom. Topol.
  \textbf{17} (2013), no.~3, 1497--1534. \MR{3073929}

\bibitem{MR0138667}
E.~S. Golod, \emph{Homologies of some local rings}, Dokl. Akad. Nauk SSSR
  \textbf{144} (1962), 479--482. \MR{0138667}

\bibitem{MR3642766}
Jes\'us Gonz\'alez, B\'arbara Guti\'errez, and Sergey Yuzvinsky, \emph{Higher
  topological complexity of subcomplexes of products of spheres and related
  polyhedral product spaces}, Topol. Methods Nonlinear Anal. \textbf{48}
  (2016), no.~2, 419--451. \MR{3642766}

\bibitem{MR4181064}
J.~Grbi\'{c} and A.~Linton, \emph{Lowest-degree triple {M}assey products in
  moment-angle complexes}, Uspekhi Mat. Nauk \textbf{75} (2020), no.~6(456),
  175--176. \MR{4181064}

\bibitem{MR3461047}
Jelena Grbi{\'c}, Taras Panov, Stephen Theriault, and Jie Wu, \emph{The
  homotopy types of moment-angle complexes for flag complexes}, Trans. Amer.
  Math. Soc. \textbf{368} (2016), no.~9, 6663--6682. \MR{3461047}

\bibitem{MR2321037}
Jelena Grbi{\'c} and Stephen Theriault, \emph{The homotopy type of the
  complement of a coordinate subspace arrangement}, Topology \textbf{46}
  (2007), no.~4, 357--396. \MR{2321037 (2008j:13051)}

\bibitem{MR3084441}
\bysame, \emph{The homotopy type of the polyhedral product for shifted
  complexes}, Adv. Math. \textbf{245} (2013), 690--715. \MR{3084441}

\bibitem{MR2138475}
E.~Grbich and S.~Terio, \emph{Homotopy type of the complement of a
  configuration of coordinate subspaces of codimension two}, Uspekhi Mat. Nauk
  \textbf{59} (2004), no.~6(360), 203--204. \MR{2138475 (2005k:55023)}

\bibitem{MR3507473}
\bysame, \emph{Homotopy theory in toric topology}, Uspekhi Mat. Nauk
  \textbf{71} (2016), no.~2(428), 3--80. \MR{3507473}

\bibitem{MR0441987}
Melvin Hochster, \emph{Cohen-{M}acaulay rings, combinatorics, and simplicial
  complexes}, Ring theory, {II} ({P}roc. {S}econd {C}onf., {U}niv. {O}klahoma,
  {N}orman, {O}kla., 1975), Dekker, New York, 1977, pp.~171--223. Lecture Notes
  in Pure and Appl. Math., Vol. 26. \MR{0441987 (56 \#376)}

\bibitem{2013arXiv1306.6221I}
K.~{Iriye} and D.~{Kishimoto}, \emph{{Topology of polyhedral products and the
  Golod property of Stanley-Reisner rings}}, ArXiv e-prints (2013).

\bibitem{2015arXiv150608970I}
\bysame, \emph{{Golodness and polyhedral products for two dimensional
  simplicial complexes}}, ArXiv e-prints (2015).

\bibitem{MR3770008}
Kouyemon Iriye and Daisuke Kishimoto, \emph{Golodness and polyhedral products
  for two-dimensional simplicial complexes}, Forum Math. \textbf{30} (2018),
  no.~2, 527--532. \MR{3770008}

\bibitem{MR3649878}
Kouyemon Iriye and Tatsuya Yano, \emph{A {G}olod complex with non-suspension
  moment-angle complex}, Topology Appl. \textbf{225} (2017), 145--163.
  \MR{3649878}

\bibitem{MR516214}
I.~M. James, \emph{On category, in the sense of {L}usternik-{S}chnirelmann},
  Topology \textbf{17} (1978), no.~4, 331--348. \MR{516214 (80i:55001)}

\bibitem{MR3627285}
Lukas Katth\"an, \emph{A non-{G}olod ring with a trivial product on its
  {K}oszul homology}, J. Algebra \textbf{479} (2017), 244--262. \MR{3627285}

\bibitem{MR3593998}
I.~Yu. Limonchenko, \emph{Families of minimally non-{G}olod complexes and
  polyhedral products}, Dalnevost. Mat. Zh. \textbf{15} (2015), no.~2,
  222--237. \MR{3593998}

\bibitem{MR3507477}
\bysame, \emph{Massey products in the cohomology of moment-angle manifolds of
  2-truncated cubes}, Uspekhi Mat. Nauk \textbf{71} (2016), no.~2(428),
  207--208. \MR{3507477}

\bibitem{LS2}
L~Lusternik and L.~Schnirelmann, \emph{Existence de trois lign\'{e}s
  g\'eod\'{e}siques ferm\'{e}es sur toute surface de genre $0$}, Comptes Rend.
  \textbf{188} (1929), 534--536.

\bibitem{LS1}
\bysame, \emph{Sur un principe topologique en analyse}, Comptes Rend.
  \textbf{188} (1929), 295--297.

\bibitem{MR531977}
Dennis McGavran, \emph{Adjacent connected sums and torus actions}, Trans. Amer.
  Math. Soc. \textbf{251} (1979), 235--254. \MR{531977 (82c:57026)}

\bibitem{MR0208458}
A.~V. Pogorelov, \emph{Regular decomposition of the {L}oba\v{c}evski\u{i}
  space}, Mat. Zametki \textbf{1} (1967), 3--8. \MR{0208458}

\bibitem{MR593648}
J.~Scott Provan and Louis~J. Billera, \emph{Decompositions of simplicial
  complexes related to diameters of convex polyhedra}, Math. Oper. Res.
  \textbf{5} (1980), no.~4, 576--594. \MR{593648}

\bibitem{MR563510}
Jan-Erik Roos, \emph{Relations between {P}oincar\'{e}-{B}etti series of loop
  spaces and of local rings}, S\'{e}minaire d'{A}lg\`ebre {P}aul {D}ubreil
  31\`eme ann\'{e}e ({P}aris, 1977--1978), Lecture Notes in Math., vol. 740,
  Springer, Berlin, 1979, pp.~285--322. \MR{563510}

\bibitem{MR1644063}
Yuli~B. Rudyak, \emph{On category weight and its applications}, Topology
  \textbf{38} (1999), no.~1, 37--55. \MR{1644063 (99f:55007)}

\bibitem{MR0356037}
Graham~Hilton Toomer, \emph{Lusternik-{S}chnirelmann category and the {M}oore
  spectral sequence}, Math. Z. \textbf{138} (1974), 123--143. \MR{0356037}

\bibitem{MR516508}
George~W. Whitehead, \emph{Elements of homotopy theory}, Graduate Texts in
  Mathematics, vol.~61, Springer-Verlag, New York-Berlin, 1978. \MR{516508}

\end{thebibliography}

\end{document}